\documentclass[a4paper,10pt]{scrartcl}

\usepackage[english]{babel}
\usepackage[fixlanguage]{babelbib}
\usepackage{cite}
\usepackage{amssymb, amsmath, amstext, amsopn, amsthm, amscd, amsxtra, amsfonts}
\usepackage{yfonts, mathrsfs}
\usepackage{stackrel}
\usepackage{paralist}

\usepackage{multicol}
\usepackage[T1]{fontenc}
\usepackage[latin9]{inputenc}
\usepackage{sectsty}
\usepackage{pstricks}
\usepackage{ifthen}
\usepackage{colonequals}

\usepackage[bookmarks,bookmarksnumbered,plainpages=false]{hyperref}

\usepackage{tikz}
\usetikzlibrary{matrix,arrows,decorations.pathmorphing}

\newtheoremstyle{standard}
 {16pt}  
 {16pt}  
 {}  
 {}  
 {\bfseries}
 {}  
 { } 
 {{\thmname{#1~}}{\thmnumber{#2.}}\thmnote{~(#3)}} 

\newtheoremstyle{kursiv}
 {16pt}  
 {16pt}  
 {\it}  
 {}  
 {\bfseries}
 {}  
 { } 
 {{\thmname{#1~}}{\thmnumber{#2.}}\thmnote{~(#3)}} 

\theoremstyle{standard}
\newtheorem{defn} [subsection]{Definition}

\newtheorem{rem}   [subsection]{Remark}

\newtheorem{setup} [subsection]{}

\theoremstyle{definition}

\theoremstyle{kursiv}
\newtheorem{thm}[subsection]{Theorem}
\newtheorem{prop} [subsection]{Proposition}
\newtheorem{cor} [subsection]{Corollary}
\newtheorem{lem} [subsection]{Lemma}

\setdefaultenum{\normalfont (a)}{i.}{A.}{1.}

\newcommand{\im}{\mathrm{im}}

\newcommand{\id}{\mathrm{id}}

\newcommand{\N}{\mathbb{N}}

\newcommand{\R}{\mathbb{R}}
\newcommand{\K}{\mathbb{K}}
\newcommand{\C}{\mathbb{C}}

\newcommand{\Fl}{Fl}
\newcommand{\set}[1]{\{  #1 \}}
\newcommand{\setm}[2]{\left\{\, #1 \middle\vert #2\,\right\}}
\newcommand{\norm}[1]{\lVert #1 \rVert}
\newcommand{\ve}{\varepsilon}
\newcommand{\coloneq}{\colonequals}
\DeclareMathOperator{\pr}{pr}
\DeclareMathOperator{\Lip}{Lip}

\title{Differentiable mappings on products with different degrees of differentiability in the two factors}
\author{Hamza Alzaareer\thanks{\small \scshape Mathematical Institute, University of Paderborn, Warburger Straße 100, 33098 Paderborn, Germany}, Alexander Schmeding\footnotemark[1]\setcounter{footnote}{-1}\thanks{Email Addresses: \href{mailto:zaareer@gmail.com}{zaareer@gmail.com} (H. Alzaareer), \href{mailto:alsch@math.upb.de}{alsch@math.upb.de} (A. Schmeding)} 
}

\begin{document}

\thispagestyle{empty}
\maketitle

\begin{abstract}
We develop differential calculus of $C^{r,s}$-mappings on products of locally convex spaces and prove exponential laws for such mappings. As an application, we consider differential equations in Banach spaces depending on a parameter in a locally convex space. Under suitable assumptions, the associated flows are mappings of class $C^{r,s}$.
\end{abstract}

\section{Introduction and statement of results}
 This paper gives a systematic treatment of the calculus of mappings on products with different degrees of differentiability in the two factors, called $C^{r,s}$-mappings. We shall develop their basic properties and some refined tools. We study such mappings in an infinite-dimensional setting, which is analogous to the approach to $C^r$-maps between locally convex spaces known as Keller's $C^r_c$-theory~\cite{keller1974} (see \cite{Mic1980}, \cite{Ham1982}, \cite{Mil1984}, \cite{hg2002a} and \cite{gn07} for streamlined expositions, cf.\ also \cite{bgn}). For $C^r$-maps on suitable non-open domains, see \cite{gn07} and \cite{wockel2006}.
Some basic facts will be recalled in Section 2).\\
 We first introduce the notion of a $C^{r,s}$-mapping: Let $E_1$, $E_2$ and $F$ be locally convex spaces, $U\subseteq E_1$ and $V\subseteq E_2$ be open subsets
and $r,s\in \N_0\cup\{\infty\}$. We say that a map $f\colon U\times V\to F$ is $C^{r,s}$ if the iterated directional derivatives
\begin{displaymath}
 (D_{(w_i,0)}\cdots D_{(w_1,0)}D_{(0,v_j)}\cdots D_{(0,v_1)}f)(x,y)
\end{displaymath}
exist for all for all $i,j\in \N_0$ with $i\leq r$ and $j\leq s$, and are continuous functions in $(x,y,w_1,\ldots,w_i,v_1,\ldots, v_j)\in U\times V\times E_1^i\times E_2^j$ (see Definition \ref{crs-map} for details).
To enable choices like $U=[0,1]$, and also with a view towards manifolds with boundary, more generally we consider $C^{r,s}$-maps if $U$ and $V$ are locally convex (in the sense
that each point has a convex neighbourhood) and have dense interior (see Definition \ref{crs1-map}). These properties are satisfied by all open sets. 
Variants and special cases of $C^{r,s}$-mappings are encountered in many parts of ana\-lysis. For example, \cite{amann1990} considers analogues of $C^{0,r}$-maps on Banach spaces based on continuous Fr\'{e}chet differentiability; see \cite[1.4]{hg2002} for $C^{0,r}$-maps; \cite{Glo2002} for $C^{r,s}$-maps on finite-dimensional domains; and \cite[p.\,135]{FaK1988} for certain $\mbox{Lip}^{r,s}$-maps in the convenient setting of ana\-lysis. Cf.\ also \cite{Nag}, \cite{Glox2} for ultrametric analogues in finite dimensions. Furthermore, a key result concerning $C^{r,s}$-maps was conjectured in \cite[p.10]{hg2004}. However the authors' interest concerning the subject was motivated by recent questions in infinite dimensional Lie theory. At the end of this section, we present an overview, showing where refined tools from $C^{r,s}$-calculus are useful.\\
The first aim of this paper is to develop necessary tools like a version of the Theorem of Schwarz and various versions of the Chain Rule. After that we turn to an advanced tool, the exponential law for spaces of mappings on products (Theorem \ref{kspace-iso}). We endow spaces of $C^r$-maps with the usual compact-open $C^r$-topology (as recalled in Definition \ref{defn: CrTop}) and spaces of $C^{r,s}$-maps with the analogous compact-open $C^{r,s}$-topology (see Definitions \ref{Crs-top} and \ref{ch-rs-in}). Recall that a topological space $X$ is called a \emph{k-space}
 if it is Hausdorff and its topology is the final topology with respect to the inclusion maps $K\to X$ of compact subsets of $X$ (cf.\ \cite{kel} and the work \cite{Ste1967}, which popularized the use of k-spaces in algebraic topology). For example, all locally compact spaces and all metrizable topological spaces are k-spaces. The main results of Section 3 (Theorems \ref{vee-rs} and \ref{kspace-iso}) subsume: 
\paragraph{Theorem A.} \emph{Let $E_1$, $E_2$ and $F$ be locally convex spaces, $U\subseteq E_1$ and $V\subseteq E_2$ be locally convex subsets with dense interior, and
$r,s\in \N_0\cup\{\infty\}$. Then $\gamma^\vee\colon U\to C^s(V,F)$, $x\mapsto \gamma(x,\bullet)$ is $C^r$ for each $\gamma\in C^{r,s}(U\times V,F)$, and the map}
\begin{equation}\label{eq:rs-iso}
 \Phi\colon C^{r,s}(U\times V,F)\to C^r(U,C^s(V,F)),\quad \gamma\mapsto \gamma^\vee
\end{equation}
\emph{is linear and a topological embedding. If $U\times V\times E_1\times E_2$ is a k-space or $V$ is locally compact, then $\Phi$ is an isomorphism of topological vector spaces.}\\[4mm]
 This is a generalization of the classical exponential law for smooth maps. Since $C^\infty$-maps and $C^{\infty,\infty}$-maps on products coincide (see Lemma \ref{lem: C11:C1}, Remark \ref{smoo} and Lemma \ref{lem: smo-smo}), we obtain as a special case that 
\begin{equation}\label{expiso}
\Phi\colon C^\infty(U\times V,F)\to C^\infty(U,C^\infty(V,F))
\end{equation}
is an isomorphism of topological vector spaces if $V$ is locally compact or $U\times V\times E_1\times E_2$ is a k-space. For open sets $U$ and $V$, the latter was known if $E_2$ is finite-dimensional or both $E_1$ and $E_2$ are metrizable (see \cite{Bil2002} and \cite{gn07}; cf. \cite[Propositions 12.2\,(b) and 12.6\,(c)]{hg2004}, where also manifolds are considered). In the inequivalent setting of differential calculus developed by E.\,G.\,F. Thomas,\footnote{Thomas replaces continuity of a function or its differentials with continuity on compact sets, and only considers quasi-complete locally convex spaces.} an exponential law for smooth functions on open sets (analogous to (\ref{expiso})) holds without any conditions on the spaces, see \cite[Theorem 5.1]{Tho1996}. Related earlier results can be found in \cite[p.\,90, Lemma 17]{Sei1972}. In the inequivalent ``convenient setting'' of analysis, (\ref{expiso})
always is an isomorphism of \emph{bornological} vector spaces (see \cite{FaK1988} and \cite{conv1997}, also for the case of manifolds) -- but rarely an isomorphism of topological vector spaces \cite{Bil2002} (in this setting other topologies on the function spaces are used). An analogue of Theorem A for finite-dimensional vector spaces over a complete ultrametric field will be
made available in \cite{Glox2}. \\ 
Naturally, one would like to apply the exponential law \eqref{eq:rs-iso} to a pair of smooth manifolds $M_1$ and $M_2$ modelled on locally convex spaces $E_1$ and $E_2$, respectively. In Section 4, we extend our results to $C^{r,s}$-maps on products of manifolds. Beyond ordinary manifolds, we can consider (with increasing generality) manifolds with smooth boundary, manifolds with corners and manifolds with rough boundary (all modelled on locally convex spaces) -- see Definition~\ref{defnNEW}. It turns out that if the modelling space of the manifold is well behaved, the exponential law holds in these cases (Theorem \ref{k-mfd-iso}). The main results of Section 4 subsume: 
 \paragraph{Theorem B.} \emph{Let $M_1$ and $M_2$ be smooth manifolds} (\emph{possibly with rough boundary}) \emph{modelled on locally convex spaces $E_1$ and $E_2$, respectively. Let $F$ be a locally convex space and $r,s\in \N_0\cup\set{\infty}$.\\ Then $\gamma^\vee\in C^r(M_1,C^s(M_2,F))$ for all $\gamma\in C^{r,s}(M_1\times M_2,F)$, and the map}
\begin{equation}\label{eq:mfd-iso}
 \Phi\colon C^{r,s}(M_1\times M_2,F)\to C^r(M_1,C^s(M_2,F)),\quad \gamma\mapsto \gamma^\vee
\end{equation}
\emph{is linear and a topological embedding. If $E_1$ and $E_2$ are metrizable, then $\Phi$ is an isomorphism of topological vector spaces.}\\[4mm] The same conclusion holds if $M_2$ is finite-dimensional or $E_1\times E_2\times E_1\times E_2$ is a k-space, provided that $M_1$ and $M_2$ are manifolds without boundary, manifolds with smooth boundary or manifolds with corners.
Recall that direct products of k-spaces need not be k-spaces. However, the direct product of two metrizable spaces is metrizable (and hence a k-space). Likewise, the product of two hemicompact k-spaces\footnote{A topological space $X$ is called hemicompact if it is the union of an ascending sequence $K_1\subseteq K_2\subseteq\cdots$ of compact sets and each compact subset of $X$ is contained in some $K_n$.} (also known as $k_\omega$-spaces) is a hemicompact k-space and hence a k-space (see \cite{FST1977} for further information and \cite{ghh}, including analogues for spaces which are only locally $k_\omega$). Thus $E_1\times E_2\times E_1\times E_2$ is a k-space whenever both $E_1$ and $E_2$ are $k_\omega$. For example, the dual $E'$ of a metrizable locally convex space $E$ always is $k_\omega$ when equipped with the compact-open topology (cf.\ \cite[Corollary 4.7]{Aus1999}). Consequently, \eqref{eq:mfd-iso} is an isomorphism in the case of manifolds with corners if $M_2$ is finite-dimensional or both $E_1$ 
and $E_2$ are metrizable resp.\ both are hemicompact $k$-spaces (Corollary \ref{mm-iso}).

As an application of the $C^{r,s}$-mappings, in Section 5 we study differential equations depending on parameters in locally convex spaces. Our approach yields the following Picard-Lindelöf type theorem: 
\paragraph{Theorem C.} \emph{Let $(E,\|.\|)$ be a Banach space, $J\subseteq \R$ be a non-degenerate interval, $F$ a locally space and $P\subseteq F$, $U\subseteq E$ be open subsets.
Let $r,s\in \N_0\cup\{\infty\}$ with $s\geq 1$ and $f\colon J\times (U\times P)\to E$ be a $C^{r,r+s}$-map. If $t_0\in J$, $z_0\in U$ and $q_0\in P$, then there exists a convex neighbourhood $J_0\subseteq J$ of $t_0$ and open neighbourhoods $U_0\subseteq U$ of $z_0$ and $P_0\subseteq P$ of $q_0$ such that for all $(\tau_0,x_0,p_0)\in J_0\times U_0\times P_0$,
the initial value problem}
  \begin{equation} \label{ode1}
   \begin{cases}
    x' (t)  &= f(t,x(t),p_0) \\
    x (\tau_0) &= x_0
   \end{cases}
  \end{equation}
\emph{has a unique solution $\phi_{\tau_0,x_0,p_0}\colon J_0\to U$, and the map}
\begin{equation}\label{partialflow}
J_0\times (U_0\times P_0)\to E,\quad (t,x_0,p_0)\mapsto \phi_{\tau_0,x_0,p_0}(t)
\end{equation}
\emph{is $C^{r+1,s}$, for each $\tau_0\in J_0$.}\\[4mm]
In particular, if $f$ is $C^{r,\infty}$, then the map in (\ref{partialflow}) is $C^{r+1,\infty}$. Using standard arguments, the local existence and differentiable dependence implies the same differentiability properties for flows of vector fields (cf. Proposition \ref{prop: flow:diff} and Proposition \ref{prop: flow:mfd}). For $f$ replaced with a $C^r$-map, the dependence of $\phi_{\tau_0,x_0,p_0}(t)$ on $(\tau_0, t, x_0,p_0)$ has already been studied in \cite[Theorem E]{hg2006}. If $F$ is a Banach space as well, then there also are classical results concerning the parameter-dependence of solutions. For example, $FC^{1,s}$-dependence (i.e., $C^{1,s}$-dependence in the sense of continuous Fr\'{e}chet differentiability) is available if $f$ is $FC^{0,s}$ (see \cite[Theorem 9.2, Remark 9.6]{amann1990}).\\[1em] \noindent
{\bf Further applications.} Several recent investigations in Lie theory (notably in the theory of infinite-dimensional Lie groups) showed the need for results as presented in this paper.  In particular, exponential laws are advanced tools for applications of $C^{r,s}$-maps in infinite dimensional Lie theory. Examples of projects which benefit from the results developed in this paper are given below. First we recall the notion of regularity:\\ If $G$ is a Lie group modelled on a locally convex space, with identity element $e$, we use the tangent map of the left translation $\lambda_g\colon G\to G$, $x\mapsto gx$ by $g\in G$ to define $g.v\coloneq T_e\lambda_g(v) \in T_g G$ for $v\in T_e(G)=:{\mathfrak g}$. Let $r\in \N_0\cup\{\infty\}$. Following \cite{Dah2011}, \cite{Glo2012}, \cite{gn07} and \cite{NaS}(cf.\ also \cite{GDL2011}), $G$ is called \emph{$C^r$-regular} if the initial value problem
\begin{displaymath}
\begin{cases}
\eta'(t)&= \eta(t).\gamma(t)\\
\eta(0) &= e
\end{cases}
\end{displaymath}
has a (necessarily unique) $C^{r+1}$-solution $\mbox{Evol}(\gamma):=\eta\colon [0,1]\to G$ for each $C^r$-curve $\gamma\colon [0,1]\to {\mathfrak g}$, and the map
\begin{displaymath}
 \mbox{evol}\colon C^r([0,1],{\mathfrak g})\to G,\quad \gamma\mapsto \mbox{Evol}(\gamma)(1)
\end{displaymath}
is smooth. If $G$ is $C^r$-regular and $r\leq s$, then $G$ is also $C^s$-regular. (If $r=\infty$, then $G$ is called \emph{regular} -- a property first defined in \cite{Mil1984}). Many results in infinite-dimensional Lie theory are only available for regular Lie groups (see \cite{Mil1984}, \cite{gn07}, \cite{Neeb2006}, \cite{NaS} and the references therein, cf.\ also \cite{KM1997}).
\begin{itemize}
\item[(a)]
Using Theorem B, one can show that the test function group $C^s_c(M,H)$ (as in \cite{hg2002}) is $C^r$-regular, for each $\sigma$-compact finite-dimensional smooth manifold $M$ and $C^r$-regular Lie group~$H$ (see \cite[p.\,270]{GDL2011} and \cite{Glox}).
\item[(b)] Let $M$ be as in (a). Using Theorems B and C, one can show that the diffeomorphism group $\mbox{Diff}(M)$ (as in \cite{Mic1980} or \cite{Glo2002}) is $C^0$-regular.
This generalizes the case of compact $M$ first obtained in \cite{OMK1982}.
\item[(c)] Likewise, Theorems B and C can be used to see that the diffeomorphism groups of $\sigma$-compact orbifolds are $C^0$-regular \cite{as2013}.
\item[(d)] Theorem A implies that $\mbox{Evol}\colon C^r([0,1],{\mathfrak g})\to C^{r+1}([0,1],G)$ is smooth, for each $C^r$-regular Lie group~$G$ \cite[Theorem A]{Glo2012}.
This result can be used to see that being $C^r$-regular is an extension property of Lie groups \cite[Theorem B.7]{NaS}. 
\item[(e)] Theorem B can be used to turn $C^{s+2}(\R,H)$ into a Lie group, for each $C^s$-regular Lie group $H$ \cite{al}.
(For the special case $s=\infty$ published in \cite{nw}, the known exponential laws for $C^\infty$-functions were sufficient). 
\item[(f)] Finally, consider a finite-dimensional Lie group $G$ and a projective limit $(\pi,E)$ of continuous Banach-representations of $G$. Theorem B has been used to see that the action $C^\infty_c(G)\times E^\infty\to E^\infty$ of the convolution algebra $C^\infty_c(G)$ of test functions on the space $E^\infty$ of smooth vectors is continuous \cite[Proposition B]{gl12} (although the conclusion becomes invalid for more general continuous representations of $G$ on locally convex spaces, see op.cit., Proposition~A).
\end{itemize} 
 
We mention that an analogous theory for $C^\alpha$-maps on products $U_1\times\cdots\times U_n$ (for $\alpha \in (\N_0\cup\{\infty\})^n$) has been presented in \cite{al}. Using an exponential law for $C^\alpha$-maps, one finds (as in \cite[III.]{wockel2006}) that every $C^\alpha$-map $f \colon I_1 \times \cdots \times I_n \rightarrow F$ from a product of closed intervals $I_1, \ldots ,I_n \subseteq \R$ to a Fr\'{e}chet space $F$ extends to a $C^\alpha$-map on $\R^n$.

\section{Preliminaries}
\numberwithin{equation}{subsection}
\textbf{Basic notation}: We write $\N \coloneq \set{1,2,3,\ldots}$ and $\N_0 \coloneq \set{0, 1, 2, \ldots}$. For a Banach space $(E,\norm{\cdot})$, we write $B_\ve^E (x) \coloneq \setm{y \in E}{\norm{x-y}< \ve}$ for $x \in E, \ve >0$. Finally for a mapping $f \colon E \rightarrow F$ which is Lipschitz-continuous we let $\Lip (f)$ be its minimum Lipschitz-constant.  

\begin{defn}\label{defn: deriv}
 Let $E, F$ be locally convex spaces, $U \subseteq E$ be an open subset and a map $f \colon U \rightarrow F$. If it exists, we define for $(x,h) \in U \times E$ the directional derivative $df(x,h) \coloneq D_h f(x) \coloneq \lim_{t\rightarrow 0} t^{-1} (f(x+th) -f(x))$. For $r \in \N_{0} \cup \set{\infty}$ we say that $f$ is $C^r$ if the iterated directional derivatives
    \begin{displaymath}
     d^{(k)}f (x,y_1,\ldots , y_k) \coloneq (D_{y_k} D_{y_{k-1}} \cdots D_{y_1} f) (x)
    \end{displaymath}
 exist for all $k \in \N_0$ such that $k \leq r$, $x \in U$ and $y_1,\ldots , y_k \in E$ and define continuous maps $d^{k} f \colon U \times E^k \rightarrow F$. If $f$ is $C^\infty$ it is also called smooth. We abbreviate $df \coloneq d^{(1)} f$.
\end{defn}

\begin{rem}
 If $E_1,E_2,F$ are locally convex topological spaces and $U \subseteq E_1,V \subseteq E_2$ open subsets together with a $C^1$-map $f \colon U \times V \rightarrow F$, then one may compute the \emph{partial derivative} $d^{(1,0)}f$ with respect to $E_1$. \\ It is defined as $d^{(1,0)} f \colon U \times V \times E_1 \rightarrow F, d^{(1,0)} f(x,y;z) \coloneq D_{(z,0)} f (x,y)$. Analogously one defines the partial derivative $d^{(0,1)} f$ with respect to $E_2$. The linearity of $df(x,y,\bullet)$ implies the so called Rule on Partial Differentials for $(x,y) \in U\times V, (h_1,h_2) \in E_1\times E_2$:   
  \begin{equation}\label{eq: rule:pdiff}
   df(x,y,h_1,h_2) = d^{(1,0} f(x,y,h_1) + d^{(0,1)} f(x,y,h_2)
  \end{equation}
 By \cite[Lemma 1.10]{hg2002a} $f \colon U \times V \rightarrow F$ is $C^1$ if and only if $d^{(1,0)}f$ and $d^{(0,1)}f$ exist and are continuous.
\end{rem}

\begin{defn}[Differentials on non-open sets]
 \begin{compactenum}
  \item The set $U \subseteq E$ is called \emph{locally convex} if every $x \in U$ has a convex neighbourhood $V$ in $U$.
  \item Let $U\subseteq E$ be a locally convex subset with dense interior. A mapping $f \colon U \rightarrow F$ is called $C^r$ if $f|_{U^\circ} \colon U^\circ \rightarrow F$ is $C^r$ and each of the maps $d^{(k)} (f|_{U^\circ}) \colon U^\circ \times E^k \rightarrow F$ admits a (unique) continuous extension $d^{(k)}f \colon U \times E^k \rightarrow F$. If $U \subseteq \R$ and $f$ is $C^{1}$, we obtain a continuous map $f' \colon U \rightarrow E, f'(x) \coloneq df(x)(1)$.\\ In particular if $f$ is of class $C^r$, we define recursively $f^{(k)} (x) = (f^{(k-1)})'(x)$ for $k \in \N_0$, such that $k \leq r$ where $f^{(0)} \coloneq f$.  
 \end{compactenum}
\end{defn}

\begin{rem}
 For the theory of $C^r$-maps, the reader is referred to \cite{hg2002a,gn07,Ham1982,Mic1980,Mil1984} (cf. also \cite{bgn}). In particular we shall use that $d^{(k)} f(x,\bullet) \colon E^k \rightarrow F$ is symmetric, $k$-linear; that compositions of composable $C^r$-maps are $C^r$; and the theorems of continuous and differentiable dependence of integrals on parameters (as recorded in \cite[Prop. 3.5]{Bil2007}). We shall also use the fact that a map $f \colon E \supseteq U \rightarrow F$ is $C^{r+1}$ if and only if $f$ is $C^1$ and $df \colon U \times E \rightarrow F$ is $C^r$.
\end{rem}

We recall the definition of the compact-open $C^r$-topology:

\begin{defn}\label{defn: CrTop}
 Let $E,F$ be locally convex topological vector spaces and $U \subseteq E$ a locally convex subset, $r \in \N_0 \cup \set{\infty}$. Denote by $C(U,F)$ the space of continuous maps from $U$ to $F$ with the compact open topology. Furthermore we denote by $C^r(U,F)$ the space of $C^r$-maps from $U$ to $F$. Endow $C^r(U,F)$ with the unique locally convex topology turning \begin{displaymath} 
   (d^{(j)} (\bullet))_{\N_0 \ni j \leq r} \colon C^{r} (U,F) \rightarrow \prod_{0 \leq j \leq r} C (U\times E^{j}, F) , f \mapsto (d^{(j)}f) 
  \end{displaymath}
 into a topological embedding. This topology is called the \emph{compact-open $C^r$ topology}. Notice that it is the initial topology with respect to the family of mappings $(d^{(j)} (\bullet))_{\N_0 \ni j \leq r}$.  
\end{defn}

\begin{defn}\label{defn: alt:CRtop}
 Let $E$ be a locally convex space, $r \in \N_0 \cup \set{\infty}$ and $J \subseteq \R$ be some non-degenerate interval. As $J$ is $\sigma$-compact we choose and fix a sequence of compact subsets $(K_n)_\N$ of $J$ with $J = \bigcup_{\N} K_n$. Fix a set $\Gamma$ of continuous seminorms which generate the topology on $E$. Let $C^r (J, E)$ the space of $C^r$-maps from $J$ to $E$ and consider on it the seminorms $\norm{\cdot}_{n,k,p}$ defined by 
   \begin{displaymath}
     \norm{\gamma}_{n,k,p} \coloneq \max_{j=0,\ldots ,k} \max_{t \in K_n} p \left(\gamma^{(j)} (t)\right)
    \end{displaymath}
  where $n \in \N$, $p \in \Gamma$ and $0\leq k \leq r$ with $k \in \N_0$. Endow $C^r(J,E)$ with the locally convex vector topology obtained from this family of seminorms.
  A variant of \cite[Proposition 4.4]{hg2002} shows that this topology is initial with respect to $d^{(j)} \colon C^r(J, E) \rightarrow C(J,E)_{\text{c.o}}, \gamma \mapsto d^{(j)} \gamma$, $0 \leq k \leq r$, i.e. it is the compact-open $C^r$-topology. 
\end{defn}

\begin{lem}\label{lem: cr:emb}
 Let $r \in \N_0 \cup \set{\infty}$, $E$ a locally convex vector space and $J$ be a non-degenerate interval. Then 
    \begin{displaymath}
     \Lambda \colon C^{r+1} (J,E) \rightarrow C(J,E) \times C^{r} (J,E) , \gamma \mapsto (\gamma , \gamma')
    \end{displaymath}
 is a linear topological embedding with closed image.
\end{lem}

\begin{proof}
 Clearly $\Lambda$ is a linear injective mapping. By the alternative description of the compact-open $C^r$-topology in \ref{defn: alt:CRtop} it is easy to see that $\Lambda$ is continuous and open onto its image. We are left to prove that $\im (\Lambda )$ is closed. To this end consider a net $(\gamma_\alpha)$ in $C^{r+1}(J,E)$ such that $\Lambda (\gamma_\alpha)$ converges in $C(J,E) \times C^r (J,E)$. Say its limit is $(\gamma , \eta )$ with $\gamma \in C(J,E)$ and $\eta \in C^r (J,E)$. Let $x\in J^\circ$. Then $\gamma_\alpha (x) \rightarrow \gamma (x)$ holds. By the fundamental theorem of calculus (\cite[Theorem 1.5]{hg2002a}) we have for $x,t \in J^\circ$: 
    \begin{equation}\label{eq: diff:int}
     \gamma_\alpha (t) =\gamma_\alpha (x) + \int_{x}^t \gamma_\alpha' (s) ds.
    \end{equation}
 Choosing $t_0$ small enough, we may assume that $[x-t_0,x+t_0]$ is contained in a finite union of the compact sets $K_n$ (see Definition \ref{defn: alt:CRtop}). As $\gamma_\alpha'$ converges to $\eta$ in the compact-open $C^r$-topology, we deduce from Definition \ref{defn: alt:CRtop} that on $[x-t_0,x+t_0]$ this net converges uniformly to $\eta$. Then \cite[Lemma 1.7]{hg2002a} implies for $t \in [x-t_0,x+t_0]$:  $\int_x^t \gamma_{\alpha } (s) ds \rightarrow \int_{x}^t \eta (s)ds$, where -for the moment- the right-hand side denotes the weak integral in the completion $\tilde{E}$ of $E$. Passing to the limit in \eqref{eq: diff:int} yields $\gamma (t) = \gamma (x) + \int_x^t \eta (s) ds$ from which we deduce that the integral coincides with $\gamma (t) - \gamma (x)$ and hence lies in $E$. Again by the fundamental theorem of calculus, we deduce that the derivative $\gamma' (t)$ exists for each $t \in [x-t_0,x+t_0]$ and coincides with $\eta (t)$. Since $x \in J^\circ$ was arbitrary, we infer that 
 $\gamma|_{J^\circ}$ is a map of class $C^1$ whose derivative equals $\eta|_{J^\circ}$. Therefore $\gamma$ is of class $C^{r+1}$ with derivative $\eta$ on $J$. Then $\lim \Lambda (\gamma_\alpha ) = (\gamma, \eta) = (\gamma, \gamma' ) = \Lambda (\gamma)$. Thus $\im (\Lambda)$ is closed. 
\end{proof}

\section{\texorpdfstring{Mappings of class {\boldmath $C^{r,s}$} on products of locally convex spaces}{Crs-mappings on products of locally convex spaces}}

\begin{defn}  \label{crs-map}
Let $E_1$, $E_2$ and $F$ be locally convex spaces, 
$U$ and $V$ open subsets of $E_1$ and $E_2$ respectively and $r,s \in \N_0 \cup \{\infty\}$. A mapping 
$ f\colon U \times V \rightarrow F $ is called a $C^{r,s}$-map, if for all $i,j \in \mathbb{N}_0$ such that 
$ i \le r, j \le s $ the iterated directional derivative 
$$d^{(i,j)}f(x,y,w_1,\dots,w_i,v_1,\dots,v_j):=(D_{(w_i,0)} \cdots D_{(w_1,0)}D_{(0,v_j)} \cdots D_{(0,v_1)}f ) (x,y)$$
exists for all $ x \in U, y \in V, w_1, \dots ,
 w_i \in E_1,  v_1, \dots ,v_j \in E_2$ and 
$$ d^{(i,j)}f: U \times V \times E^i_1 \times E^j_2 \rightarrow F ,$$
$$  \quad (x,y,w_1,\dots,w_i,v_1,\dots,v_j)\mapsto (D_{(w_i,0)} \cdots D_{(w_1,0)}D_{(0,v_j)} \cdots D_{(0,v_1)}f ) (x,y)$$
is continuous.
\end{defn}

More generally, it is useful to have a definition of $C^{r,s}$-maps on not necessarily open domains available:

\begin{defn}  \label{crs1-map}
Let $E_1$, $E_2$ and $F$ be locally convex spaces, 
$U$ and $V$ are locally convex subsets with dense interior of $E_1$ and $E_2,$ respectively, and $r,s \in \N_0 \cup \{\infty\}$, then 
we say that $ f\colon U \times V \rightarrow F $ is a $C^{r,s}$-map, if $f \lvert_{U^0 \times V^0} \colon {U^0 \times V^0}\rightarrow F$ is $C^{r,s}$-map and for all $i,j \in \N_0$ such that $ i \le r, j \le s $, the map
$$ d^{(i,j)}(f\lvert_{U^0 \times V^0}) \colon U^0 \times V^0 \times E^i_1 \times E^j_2 \rightarrow F $$
admits a continuous extension 
 $$d^{(i,j)}f\colon U \times V \times E^i_1 \times E^j_2 \rightarrow F .$$
\end{defn}

Definitions \ref{crs-map} and \ref{crs1-map} can be rephrased as follows:

\begin{lem}\label{fx-dij}
Let $E_1$, $E_2$ and $F$ be locally convex spaces, 
$U$ and $V$ be locally convex subsets with dense interior of $E_1$ and $E_2$ respectively and $r,s \in \N_0 \cup \{\infty\}$. Then 
$ f \colon U \times V \rightarrow F $ is $C^{r,s}$-map if and only if all of the following conditions are satisfied:
 \begin{compactenum}
\item For each $ x \in U$, the map $ f_x\coloneq f(x,\bullet)\colon V \rightarrow F, \; y \mapsto f_x(y)\coloneq f(x,y)$ is $C^s$.
\item For all $ y \in V$ and $j \in \N_0$ such that $j \leq s$ and $v \coloneq v_1,\dots,v_j \in E_2$, the map 
$d^{(j)}f_\bullet(y,v)\colon U \rightarrow F, \; x \mapsto (d^{(j)}f_x)(y,v)$ is $C^r$.
\item $d^{(i,j)}f\colon U \times V \times E^{i}_{1} \times E^{j}_{2} \rightarrow F, \; (x,y,w,v)\mapsto d^{(i)}(d^{(j)}f_\bullet(y,v))(x,w)$ is continuous, for all $j$ as in {\rm (b)}, $i \in \N_0$ such that $i \leq r$ and $w\coloneq (w_1,\ldots, w_i) \in E^i_1$.
\end{compactenum}
\end {lem}

\begin{proof}
\textit{Step 1.} If $U,V$ are open subsets, then the equivalence is clear.\newline
Now the general case: Assume that $f$ is a $C^{r,s}$-map.\newline
\textit{Step 2.} If $x \in U^0 ,$ then for $j\in \N_0, \, j \leq s$
$$ D_{(0,v_j)}\cdots D_{(0,v_1)}f(x,y)=D_{v_j} \cdots D_{v_1}f_x(y)$$
exists for all $y \in V^0 $ and $v_1, \ldots,v_j \in E_2$, with continuous extension 
$$(y,v_1,\ldots,v_j) \mapsto d^{(0,j)}f(x,y,v_1,\ldots,v_j)$$
to  $V\times E^j_2 \rightarrow F.$ Hence $f_x\colon V\rightarrow F$ is $C^s.$ \newline
If $x \in U$ is arbitrary, $y \in V^0$ and $v_1 \in E_2,$ we show that $D_{v_1}f_x(y)$ exists and equals $d^{(0,1)}f(x,y,v_1)$. There exists $R >0$ such that $y+tv_1 \in V$ for all $t \in \R$, $\left|t\right| \leq R$ and there exists a relatively open convex neighbourhood $W \subseteq U$ of $x$ in $U.$
Because $U^0$ is dense, there exists $z \in U^0 \cap W.$ Since $W$ is convex, we have $x+\tau(z-x) \in W$ for all $\tau \in \left[0,1\right]$. Moreover, since $z \in W^0$ we have, $x + \tau(z-x)\in W^0$ for all $\tau \in \left(0,1\right] \subseteq U^0.$ Hence, for $\tau \in \left(0,1\right], \; f(x+\tau(z-x),y)$ is $C^s$ in $y$, and thus for $t \neq 0$
$$\frac{1}{t}(f(x+\tau(z-x),y+tv_1)-f(x+\tau(z-x),y))=\int_0^1 d^{(0,1)}f(x+\tau(z-x),y+ \sigma tv_1,v_1)\,d\sigma$$
by the Mean Value Theorem.
Now let $\tilde{F}$ be a completion of $F$. Because 
$$ h\colon \left[0,1\right] \times \left[-R,R\right] \times \left[0,1\right] \rightarrow \tilde{F} , \; (\tau ,t,\sigma)\mapsto d^{(0,1)}f(x+\tau(z-x),y+\sigma tv_1,v_1)$$
is continuous, also the parameter-dependent integral 
$$g\colon \left[0,1\right] \times \left[-R,R\right] \rightarrow \tilde{F}, \; g(\tau,t):=\int_0^1h(\tau,t,\sigma)\,d\sigma $$
is continuous. Fix $t \neq 0$ in $\left[-R,R\right].$ Then 
\begin{equation} \label{eq:mvt-tau}
g(\tau,t)= \frac{1}{t}(f(x+\tau(z-x),y+tv_1)-f(x+\tau(z-x),y))
\end{equation}
for all $\tau \in \left(0,1\right].$ By continuity of both sides in $\tau,$ (\ref{eq:mvt-tau}) also holds for $\tau =0.$
Hence 
$$ \frac{1}{t}(f(x,y+tv_1)-f(x,y))=g(0,t) \rightarrow g(0,0)$$
as $t\rightarrow 0.$ Thus $D_{v_1}f_x(y)$ exists and is given by
\begin{align*}
g(0,0)&=\int_0^1 d^{(0,1)}f(x,y,v_1)\, d \sigma 
       = d^{(0,1)}f(x,y,v_1).
\end{align*}
Holding $v_1$ fixed, we can repeat the argument to see that $D_{v_j}\cdots D_{v_1}f_x(y)$ exists for all $y \in V^0$ and $j \in \N_0$ such that $j\leq s$ and all $v_1,\ldots,v_j \in E_2,$ and is given by 
$$D_{v_j}\cdots D_{v_1}f_x(y) =d^{(0,j)}f(x,y,v_1,\ldots,v_j).$$
Since the right-hand side makes sense for $(y,v_1,\ldots,v_j) \in V \times E^j_2$ and is continuous there, $f_x$ is $C^s.$\newline
\textit{Step 3} Holding $v_1, \ldots,v_j \in E_2^j$ fixed, the function $(x,y) \mapsto d^{(0,j)}f(x,y,v_1,\ldots,v_j)$ is $C^{r,0}.$ By Step 2 (applied to the $C^{(0,r)}$ function $(y,x)\mapsto d^{(0,j)}f(x,y,v_1,\ldots,v_j)$) we see that for each $y \in V,$ the function $ U \rightarrow F, \; x \mapsto d^{(0,j)}f(x,y,v_1,\ldots,v_j)$ is $C^r $ and  $d^{(i)}(d^{(j)}f_\bullet(y,v))(x,w) = d^{(i,j)}f(x,y,w,v),$ which is continuous in $(x,y,w,v)\in U\times V \times E_1^i \times E_2^j.$ Hence if $f$ is $C^{r,s},$ then (a),(b) and (c) hold.

\textit{Step 1.} Conversely. Assume that (a),(b) and (c) hold. By Step 1, $f\lvert _{U^0 \times V^0}$ is $C^{r,s}$ and 
\begin{equation} \label{eq:dij-ddj}
d^{(i,j)}f\lvert_{U^0 \times V^0}(x,y,w,v)=d^{(i)}(d^{(j)}f_\bullet(y,v))(x,w) 
\end{equation}
for $(x,y) \in U^0 \times V^0, \, w \in E_1^i, \, v \in E_2^j.$
By (c), the right-hand side of (\ref{eq:dij-ddj}) extends to a continuous function $d^{(i,j)}f\colon U \times V \times E_1^i \times E^j_2 \rightarrow F.$ Hence $f$ is a $C^{r,s}$-map.   
\end{proof}
 
The following lemma will enable us to prove a version of the Theorem of Schwarz for $ C^{r,s}$-maps.

\begin{lem} \label{ini-sch}
Let $E_1,\, E_2$ and $F$ be locally convex spaces, 
$ f\colon U \times V \rightarrow F $ be a $C^{1,1}$-map on open subsets $U \subseteq E_1,\, V \subseteq E_2$ and $w \in E_1, \, v \in E_2$ such that $D_{(w,0)}D_{(0,v)}f$ exists and is continuous as a map $U \times V \to F.$ Then also $D_{(0,v)}D_{(w,0)}f$ exists and coincides with $D_{(w,0)}D_{(0,v)}f$.  
\end{lem}

\begin{proof}
After replacing $F$ with a completion, we may assume that $F$ is complete. Fix $x \in U,\, y \in V.$ There is $\varepsilon >0$ such that $x+sw \in U$ and $y+tv \in V$ for all $s,t \in B^{\R}_{\varepsilon}(0).$ For $t \neq 0$ as before, we have
\begin{equation} \label{eq:swtv}  
\frac {1}{t}(f(x+sw , y+tv)- f(x+sw, y))=\int^{1}_{0} D_{(0,v)}f (x+sw,y+rtv)dr.
\end{equation}
For fixed $t$, consider  the map 
$$ g\colon B^{\R}_{\varepsilon}(0) \to F, \, g(s)\coloneq\int^{1}_{0} D_{(0,v)}f (x+sw,y+rtv)dr.$$
The map $[0,1] \times B^{\R}_{\varepsilon}(0) \to F, \, (r,s) \mapsto D_{(0,v)}f (x+sw,y+ rt) $ is differentiable in $s,$ with partial derivative $D_{(w,0)}D_{(0,v)}f(x+sw,y+rtv)$ which is continuous in $(r,s).$ Hence, by \cite[Proposition 3.5]{Bil2007}, $g$ is $C^1$ and 
$$g'(0)= \int^{1}_{0}  D_{(w,0)}D_{(0,v)}f(x,y+rtv)dr.$$ Hence (\ref{eq:swtv}) can be differentiated with respect to $s,$ and 
\begin{equation} \label{eq:swtv1}  
\frac {1}{t}(D_{(w,0)}f(x, y+tv)- D_{(w,0)}f(x, y))=\int^{1}_{0}  D_{(w,0)}D_{(0,v)}f(x,y+rtv)dr.
\end{equation}
Note that, for fixed $x$, $v$ and $w,$ the integrand in (\ref{eq:swtv1}) also makes sense for $t=0,$ and defines a continuous function $h\colon [0,1] \times B^{\R}_{\varepsilon}(0)\to F$ of $(r,t).$  
By \cite[Proposition 3.5]{Bil2007}, the function 
$$H\colon B^{\R}_{\varepsilon}(0) \to F,\, H(t):=\int^{1}_{0} h(r,t)dr$$
is continuous. If $t\neq0$ this function coincides with $\frac {1}{t}(D_{(w,0)}f(x, y+tv)- D_{(w,0)}f(x, y)),$ by (\ref{eq:swtv1}). Hence 
\begin{align*}
&D_{(0,v)}D_{(w,0)}f(x,y)=\lim_{t\to 0}\frac{1}{t}(D_{(w,0)}f(x, y+tv)- D_{(w,0)}f(x, y))\\
&= \lim_{t\to 0}H(t)=H(0)=\int^{1}_{0}h(r,0)dr=D_{(w,0)}D_{(0,v)}f(x,y)
\end{align*}
exists and has the asserted form. 
\end{proof}

\begin{lem}\label{semi-schwarz}
Let $E_1$, $E_2$ and $F$ be locally convex spaces, 
$U$ and $V$ be open subsets of $E_1$ and $E_2,$ respectively, and $r \in \N_0 \cup \{\infty\}$. If
$ f\colon  U \times V \rightarrow F $ is a $C^{r,1}$-map, then 
$$D_{(0,v)} D_{(w_i,0)} \cdots D_{(w_1,0)}f(x,y)$$
exists for all $i \in \N$ such that $i \leq r,\;(x,y)\in U \times V ,\; v \in E_2$ and $w_1,\ldots, w_i \in E_1,$ and it coincides with  $d^{(i,1)}f(x,y,w_1,\ldots,w_i,v).$
\end{lem}
\begin{proof}
The proof is by induction on $i.$ \textit{The case $i=1.$} This is covered by Lemma \ref{ini-sch}.\newline
\textit{Induction step.} Assume that $i>1$. By induction, we know that 
$$D_{(0,v)} D_{(w_{i-1},0)} \cdots D_{(w_1,0)}f(x,y)$$
exists and coincides with 
\begin{equation} \label{eq:imin-o} 
d^{(i-1,1)}f(x,y,w_1,\ldots,w_{i-1},v).
\end{equation}
Define $g\colon U \times V \rightarrow F$ via 
$$g(x,y)=D_{(w_{i-1},0)} \cdots D_{(w_1,0)}f(x,y)=d^{(i-1,0)}f(x,y,w_1,\ldots,w_{i-1}).$$
Then $g$ is $C^{(1,0)}$ ($f$ is $C^{r,1}$ and $r \geq i,$ hence we can differentiate once more in the first variable).
By induction, $g$ is differentiable in the second variable with
\begin{align} 
\label{gd-dg}
 D_{(0,v)}g(x,y) &= d^{(i-1,1)}f(x,y,w_1,\ldots,w_{i-1},v)\\ 
\label{gd-dg1} 
                 &= D_{(w_{i-1},0)} \cdots D_{(w_1,0)}D_{(0,v)}f(x,y),
\end{align}
which is continuous in $(v,x,y).$ Hence $g$ is $C^{0,1}.$ and $d^{(0,1)}g(x,y,v)$ is given by (\ref{eq:imin-o}). Because $f$ is $C^{r,1}$ and $r \geq i,$ the right-hand side of (\ref{gd-dg}) can be differentiated once more in the first variable, hence also $D_{(0,v)}g(x,y),$ with 
\begin{align*}
d^{(1,1)}g(x,y,w_i,v) = D_{(w_i,0)}D_{(0,v)}g(x,y)&= D_{(w_i,0)}\cdots D_{(w_1,0)}D_{(0,v)}f(x,y)\\
                                                  &= d^{(i,1)}f(x,y,w_1,\ldots,w_{i},v).
\end{align*}
As this map is continuous, $g$ is $C^{1,1}.$ By Lemma \ref{ini-sch}, also $D_{(0,v)}D_{(w_i,0)}g(x,y)$ exists and is given by 
$D_{(w_i,0)}D_{(0,v)}g(x,y)= d^{(i,1)}f(x,y,v,w_1,\ldots,w_{i})$
(where we used (\ref{gd-dg1})). But, by definition of $g,$ 
$$D_{(0,v)}D_{(w_i,0)}g(x,y)=D_{(0,v)}D_{(w_i,0)}D_{(w_{i-1},0)}\cdots D_{(w_1,0)}f(x,y).$$
Hence $D_{(0,v)}D_{(w_i,0)} \cdots D_{(w_1,0)}f(x,y)=d^{(i,1)}f(x,y,v,w_1,\ldots,w_{i}).$
\end{proof}

\begin{thm}[Schwarz' Theorem] \label{schwarz}
Let $E_1$, $E_2$ and $F$ be locally convex spaces and 
$ f\colon U \times V \rightarrow F $ be $C^{r,s}$-map on open subsets 
$U \subseteq E_1,\, V \subseteq E_2.$ Let $ x \in U, \, y \in V,\,w_1, \dots 
,w_i \in E_1$ and $ w_{i+1},\dots,w_{i+j} \in   E_2.$ Define $w_k^*:=(w_k,0)$ if $k \in \{1,\ldots,i\}$ and $w^*_k:=(0,w_k)$ if $k \in \{i+1,\ldots,i+j\}.$ Let $i,j \in \mathbb{N}_0 $ with $ i\le r,j \le s $ and $ \sigma  \in S_{i+j}  $  be a permutation of 
$\left\{1,\dots,i+j\right\}.$ Then the iterated directional derivative 
$$(D_{w^*_{\sigma(1)}} \cdots D_{w^*_{\sigma(i+j)}}f ) (x,y)$$
exists and coincides with
$$d^{(i,j)}f(x,y,w_1,\dots,w_i,w_{i+1},\dots,w_{i+j}).$$
\end{thm} 
                                   
\begin{proof}
The proof is by induction on $i+j.$ \textit{The case $i+j=0$} is trivial.\newline
\textit{The case $i=0$ or $j=0$.} If $i=0,$ then the assertion follows from Schwarz Theorem for the $C^s$-function $f(x,\bullet)\colon  V\rightarrow F .$ Likewise if  $j=0,$ then the assertion follows from Schwarz Theorem for the $C^r$-function $f(\bullet,y)\colon  U \rightarrow F .$ (see \cite{gn07}).\newline
\textit{The case $i,j \geq 1.$} If $\sigma(1) \in \{1,\ldots,i\},$ then by induction,
\begin{align*}
&D_{w^*_{\sigma(2)}}\cdots D_{w^*_{\sigma(i+j)}}f(x,y)\\
&=d^{(i-1,j)}f(x,y,w_1,\ldots,w_{\sigma(1)-1},w_{\sigma(1)+1},\ldots,w_i,w_{i+1},\ldots,w_{i+j}).
\end{align*}
Because $f$ is $C^{i,j},$ we can differentiate once more in first variable:
\begin{align*}
&D_{w^*_{\sigma(1)}}\cdots D_{w^*_{\sigma(i+j)}}f(x,y)\\ &=d^{(i,j)}f(x,y,w_1,\ldots,w_{\sigma(1)-1},w_{\sigma(1)+1},\ldots,w_i,w_{\sigma(1)};w_{i+1},\ldots,w_{i+j})\\
&= d^{(i,j)}f(x,y,w_1,\ldots,w_i;w_{i+1},\ldots,w_{i+j}).
\end{align*}
For the final equality we used that
$$d^{(i,j)}f(x,y,z_1,\ldots,z_i,v_1,\ldots,v_j)=d^{(i)}(d^{(j)}f_\bullet(y,v_1,\ldots,v_j))(x,z_1,\ldots,z_j)$$
is symmetric in $z_1,\ldots,z_i$, as $g(x):=d^{(j)}f_x(y,v_1,\ldots,v_j)$ is $C^r$ in $x$ (see Lemma \ref{fx-dij}). \newline
If $\sigma(1) \in \{i+1,\ldots,i+j\},$ then by induction,
$$D_{w^*_{\sigma(2)}}\cdots D_{w^*_{\sigma(i+j)}}f(x,y)=d^{(i,j-1)}f(x,y,w_1,\ldots,w_i,\ldots,w_{\sigma(1)-1},w_{\sigma(1)+1},\ldots,w_{i+j}).$$
For fixed $w_{i+1},\ldots,w_{i+j},$ consider the function $h\colon U\times V\rightarrow F,$
$$h(x,y):=d^{(0,j-1)}f(x,y, w_{i+1},\ldots,w_{\sigma(1)-1},w_{\sigma(1)+1},\ldots,w_{i+j}),$$
which is $C^{r,s-(j-1)}.$\newline
By Lemma \ref{semi-schwarz}, 
$$D_{w^*_{\sigma(1)}}D_{w^*_i}\cdots D_{w^*_1}h(x,y)$$
exists and coincides with
$$D_{w^*_i}\cdots D_{w^*_1}D_{w^*_{\sigma(1)}}h(x,y).$$

Now
\begin{align*}
D_{w^*_{\sigma(2)}}\cdots D_{w^*_{\sigma(i+j)}}f(x,y)&= d^{(i,j-1)}f(x,y,w_1,\ldots,w_{\sigma(1)-1},w_{\sigma(1)+1},\ldots,w_{i+j})\\
                                             &= D_{w^*_i}\cdots D_{w^*_1}h(x,y).
\end{align*} 
By the preceding, we can apply, $D_{w^*_{\sigma(1)}}$, i.e., $D_{w^*_{\sigma(1)}}\cdots D_{w^*_{\sigma(i+j)}}f(x,y)$ exists and coincides with
\begin{align*}
&D_{w^*_i}\cdots D_{w^*_1}D_{w^*_{\sigma(1)}}h(x,y)\\
&= d^{(i,j)}f(x,y,w_1,\ldots,w_i,w_{i+1},\ldots,w_{\sigma(1)-1},w_{\sigma(1)+1},\ldots,w_{i+j},w_{\sigma(1)})\\
&=d^{(i)}(d^{(j)}f_\bullet(y,w_{i+1},\ldots,w_{\sigma(1)-1},w_{\sigma(1)+1},\ldots,w_{i+j},w_{\sigma(1)}))(x,w_1,\ldots,w_i)
\end{align*}
where $d^{(j)}f_x(y,v_1,\ldots,v_j)$ is symmetric in $v_1,\ldots,v_j$ by the Schwarz Theorem for the $C^s$-function $f_x.$
Hence 
$$d^{(j)}f_x(y,w_{i+1},\ldots,w_{\sigma(1)-1},w_{\sigma(1)+1},\ldots,w_{i+j},w_{\sigma(1)})= d^{j}f_x(y,w_{i+1},\ldots,w_{i+j})$$
for all $x.$ Hence also after differentiations in $x$:
$$d^{(i)}(d^{(j)}f_\bullet(y,w_{i+1},\ldots,w_{\sigma(1)-1},w_{\sigma(1)+1},\ldots,w_{i+j},w_{\sigma(1)}))(x,w_1,\ldots,w_i)$$
coincides with 
$$d^{(i,j)}f(x,y,w_1,\ldots,w_{i+j})=d^{(i)}(d^{(j)}f_\bullet(y,w_{i+1},\ldots,w_{i+j}))(x,w_1,\ldots,w_i).$$ 

\end{proof}

\begin{rem} \label{sdint}
If $U$ and $V$ are merely locally convex subsets with dense interior in the situation of Theorem \ref{schwarz}, then 
\begin{equation} \label{eq:sch-sdi}
(D_{w^*_{\sigma(1)}}\cdots D_{w^*_{\sigma(i+j)}}f)(x,y)
\end{equation}
exists for all $x\in U^0, \, y \in V^0,$ and the map $d^{(i,j)}f(x,y,w_1,\ldots,w_{i+j})$ provides a continuous extension of (\ref{eq:sch-sdi}) to all of $U \times V \times E^i_1 \times E^j_2.$
\end{rem}

\begin{cor}\label{rs-sr}
Let $E_1$, $E_2$ and $F$ be locally convex spaces, 
$U$ and $V$ be locally convex subsets with dense interior of $E_1$ and $E_2$ respectively. If $f\colon  U \times V \rightarrow F $ is $C^{r,s},$ then $$g\colon V\times U \rightarrow F, \, (y,x) \mapsto f(x,y)$$
is a $C^{s,r}$-map, and 
$$d^{(j,i)}g(y,x,v_1,\ldots,v_j,w_1,\ldots,w_i)=d^{(i,j)}f(x,y,w_1,\ldots,w_i,v_1,\ldots,v_j)$$
for all $i,j \in \N_0$ with $i \leq r, \, j \leq s, \, x \in U, y \in V,\, w_1,\ldots,w_i \in E_1$ and $v_1,\ldots,v_j \in E_2.$ 
\end{cor}

\begin{lem} \label{lam-rs}
Let $E_1$, $E_2$ and $F$ be locally convex spaces, 
$U$ and $V$ be locally convex subsets with dense interior of $E_1$ and $E_2$ respectively. If $f\colon  U \times V \rightarrow F $ is $C^{r,s}$ and $\lambda\colon  F \rightarrow H$ is a continuous linear map to a locally convex space $H,$ then $\lambda \circ f $ is $C^{r,s}$ and $d^{(i,j)}(\lambda \circ f)= \lambda \circ d^{(i,j)}f.$
\end{lem}

\begin{proof}
Follows from the fact that directional derivatives and continuous linear maps can be interchanged.
\end{proof}

\begin{lem}[Mappings to products] \label{lem: mp:prod} Let $E_1, E_2$ be  locally convex spaces, $U$ and $V$ be locally convex subsets with dense interior of $E_1$ and $E_2$ respectively, and $(F_\alpha)_{\alpha \in A}$ be a family of locally convex spaces with direct product $F:= \prod_{\alpha \in A }F_{\alpha}$ and the projections $\pi_\alpha \colon  F \rightarrow F_\alpha$ onto the components.
Let $r,s \in \N_0 \cup \{\infty \}$ and $f\colon U \times V \rightarrow F$ be a map. Then $f$ is $C^{r,s}$ if and only if all of its components $f_\alpha \coloneq = \pi_\alpha \circ f$ are $C^{r,s}$. In this case
\begin{equation} \label{eq:mip}
d^{(i,j)}f =(d^{(i,j)}f_\alpha)_{\alpha\in A}.
\end{equation}
for all $i,j \in \N_0$ such that $i \leq r$ and $ j \leq s.$

\end{lem}

\begin{proof}
$\pi_\alpha$ is continuous linear. Hence if $f$ is $C^{r,s}$, then $f_\alpha= \pi_\alpha \circ f$ is $C^{r,s}$, by Lemma \ref{lam-rs}, with $d^{(i,j)}f_\alpha =\pi_\alpha \circ d^{(i,j)}f.$
Hence (\ref{eq:mip}) holds.

Conversely, assume that each $f_{\alpha}$ is $C^{r,s}.$ Because the limits in products can be formed component-wise, we see that 
$$d^{(i,j)}f(x,y,w_1,\ldots,w_i,v_1,\ldots,v_j)=D_{(w_i,0)}\cdots D_{(w_1,0)}D_{(0,v_j)}\cdots D_{(0,v_1)}f(x,y)$$
exists for all $(x,y) \in U^0 \times V^0$ and $w_1,\ldots,w_i \in E_1,\quad v_1, \ldots,v_j \in E_2,$ and is given by 
\begin{equation} \label{eq:dfa}
(d^{(i,j)}f_\alpha(x,y,w_1,\ldots, w_i, v_1,\ldots, v_j))_{\alpha \in A}.
\end{equation}

Now (\ref{eq:dfa}) defines a continuous function $ U \times V \times E^i_1 \times E^j_2 \rightarrow F $
for all $i,j \in \N_0$ such that $i \leq r$ and $j \leq s.$ Hence $f$ is $C^{r,s}.$
\end{proof}

\begin{lem} \label{rn-n1}
Let $r,s \in \N_0 \cup \{ \infty \},\, s\geq 1$, $E_1, E_2,F$ be  locally convex spaces, $U$ and $V$ be locally convex subsets with dense interior of $E_1$ and $E_2$ respectively. Let $f\colon U \times V \to F$ be a map. Then $f$ is $C^{r,s}$ if and only if $f$ is $C^{r,0},\; f$ is $C^{0,1}$ and $d^{(0,1)}f\colon U\times (V \times E_2 )\rightarrow F$ is $C^{r,s-1}.$  
\end{lem}

\begin{proof}
The implication $``\Rightarrow"$ will be established after Lemma \ref{rs-hh}, and shall not be used before. To prove $``\Leftarrow",$ let $i,j \in \N_0$ such that $i \leq r$ and $ j \leq s,$ and $(x,y) \in U^0 \times V^0$ and $w_1, \ldots, w_i \in E_1$ and $v_1, \ldots, v_j \in E_2.$
\newline
\textit{If $j=0$}, then $D_{(w_i,0)} \cdots D_{(w_1,0)}f(x,y)$ exists as $f$ is $C^{r,0},$ and is given by 
$$d^{(i,0)}f(x,y,w_1,\ldots,w_i)$$
 which extends continuously to $U \times V \times E^{i}_1.$
\newline
\textit{If $j>0$}, then $D_{(0,v_1)}f(x,y)= d^{(0,1)}f(x,y,v_1)$ exists because $f$ is $C^{0,1}$ and because this function is $C^{r,s-1},$ also the directional derivatives
$$D_{(w_i,0)} \cdots D_{(w_1,0)}D_{(0,v_j)} \cdots D_{(0,v_1)}f(x,y)$$
$$=D_{(w_i,(0,0))} \cdots D_{(w_1,(0,0))}D_{(0,(v_j,0))} \cdots D_{(0,(v_2,0))}(d^{(0,1)}f)(x,y,v_1)$$ 
exist and the right-hand side extends continuously to $(x,y,w_1,\ldots,w_i,v_1,\ldots,v_j)\in U\times V \times E^i_1 \times E^j_2.$ Hence $f$ is $C^{r,s}.$    
\end{proof} 

\begin{lem} \label{rs-la}
Let $r,s \in \N_0 \cup \{ \infty \}$, $E_1, E_2,H_1,H_2, F$ be locally convex spaces, $U,V,P$ and $Q$ be locally convex subsets with dense interior of $E_1,E_2, H_1$ and $H_2,$ respectively. If $ f\colon  U \times V  \rightarrow F $ is a $C^{r,s}$-map and $\lambda_1\colon H_1 \rightarrow E_1$ as well as $\lambda_2\colon H_2 \rightarrow E_2$ are continuous linear maps such that $\lambda_1(P) \subseteq U$ and $ \lambda_2(Q)\subseteq V$, then $ f \circ (\lambda_1 \times \lambda_2)\lvert_{P\times Q}\colon  P \times Q \rightarrow F$ is $C^{r,s}$.
\end{lem}
\begin{proof}
For $(q,p)\in P^0 \times Q^0$ and $w_1,\ldots,w_i \in H_1, \; v_1,\ldots,v_j \in H_2,$ we have
\begin{align*}
&D_{(0,v_1)}(f \circ (\lambda_1 \times \lambda_2))(q,p) = \lim_{t \rightarrow 0} \frac{1}{t}(f(\lambda_1(p),\lambda_2(q+tv_1)-f(\lambda_1(p),\lambda_2(q)))\\
&=\frac{1}{t}(f(\lambda_1(p),\lambda_2(q)+t\lambda_2( v_1))-f(\lambda_1(p),\lambda_2(q)))                                                       =(D_{(0,\lambda_2(v_1))}f)(\lambda_1(p),\lambda_2(q))
\end{align*}
and recursively
$$ D_{(w_i,0)}\cdots D_{(w_1,0)}D_{(0,v_j)}\cdots D_{(0,v_1)}(f\circ(\lambda_1 \times \lambda_2))(p,q)$$
$$=d^{(i,j)}f(\lambda_1(p),\lambda_2(q),\lambda_1(w_1),\ldots,\lambda_1(w_i),\lambda_2(v_1),\ldots,\lambda_2(v_j)).$$
The right-hand side defines a continuous function of $(p,q,w_1, \ldots,w_i,v_1,\ldots,v_j)\in$\\$ P \times Q \times H^i_1 \times H^j_2.$ Hence the assertion follows.
\end{proof}

\begin{lem} \label{rs-hh}
Let $r,s \in \N_0 \cup \{ \infty \}$, $E_1, E_2,H_1, \ldots, H_n, F$ be  locally convex spaces, $U$ and $V$ be locally convex subsets with dense interior of $E_1$ and $E_2,$ respectively, and 
$$ f\colon  U \times V \times H_1 \times \cdots \times H_n \rightarrow F $$
be a map with the following properties:

\begin{compactenum}
\item $f(x,y,\bullet)\colon  H_1 \times \cdots \times H_n \rightarrow F$ is $n$-linear for all $x\in U, \; y \in V;$
\item The directional derivatives $D_{(w_i,0,0)} \cdots D_{(w_1,0,0)}D_{(0,v_j,0)} \cdots D_{(0,v_1,0)}f(x,y,h)$ exist for all $ i,j  \in \N_0$ such that $i \leq r, j \leq s$, $(x,y) \in U^0\times V^0$, $h \in H_1 \times \cdots \times H_n$ and
 $w_1,\ldots,w_i \in E_1, \; v_1,\ldots,v_j \in E_2,$ and extend continuously to functions 
$$  U \times V \times H_1 \times \cdots \times H_n \times E^{i}_{1} \times E^{j}_{2} \rightarrow F.$$
Then $f\colon  U \times (V \times H_1 \times \cdots \times H_n) \rightarrow F$
is $C^{r,s}.$ Also 
$g\colon  (U  \times H_1 \times \cdots \times H_n ) \times V \rightarrow F, \, ((x,h),y)\mapsto f(x,y,h)$
is $C^{r,s}.$
 \end{compactenum}
\end{lem}

\begin{proof}
Holding $h \in H:= H_1 \times \cdots \times H_n$ fixed, the map $f(\bullet,h)$ is $C^{r,s}$ and hence 
$$ \varphi\colon  V \times U \rightarrow F, \; (x,y) \mapsto f(y,x,h)$$
is $C^{s,r}$, by Corollary \ref{rs-sr}, with 
$$ D_{(0,v_j)} \cdots D_{(0,v_1)} D_{(w_i,0)} \cdots D_{(w_1,0)}\varphi(x,y)$$

$$=D_{(w_i,0)} \cdots D_{(w_1,0)} D_{(0,v_j)} \cdots D_{(0,v_1)}f(y,x,h)$$
Hence $f_1\colon V\times (U\times H) \to F,\, f_1(y,x,h):=f(x,y,h)$ satisfies hypotheses analogues to those for $f$ (with $r$ and $s$ interchanged) and will be $C^{s,r}$ if the first assertion holds. Using Corollary \ref{rs-sr}, this implies that $g$ is $C^{r,s}.$ Hence we only need to prove the first assertion.

We may assume that $r,s < \infty;$ the proof is by induction on $s$. 
\newline
\textit{The case $s=0$}. Then $f$ is $C^{r,0}$ by the hypotheses.
\newline
\textit{Induction step}. Let $v \in E_2, z=(z_1, \ldots,z_n) \in H.$ By hypothesis, $D_{(0,v,0)}f(x,y,h)$ exists and extends to a continuous map on $U \times V \times E_2 \times H \rightarrow F.$ Because $f(x,y,\bullet)\colon  H \rightarrow F$ is continuous and $n$-linear, it is $C^1$ with
\begin{equation*} 
D_{(0,0,z)}f(x,y,h)= \sum_{k=1}^n f(x,y,h_1,\ldots,h_{k-1},z_k,h_{k+1},\ldots,h_n).
\end{equation*}
This formula defines a continuous function $U \times V \times H \times H \rightarrow F.$
Holding $x \in U$ fixed, we deduce with the Rule on Partial Differentials (Definition \ref{defn: deriv}) that the map 
$$ V \times H \rightarrow F, (x,h) \mapsto f(x,y,h)  $$
is $C^1,$ with 
\begin{equation} \label{d-sum}
D_{(0,v,z)}f(x,y,h)=D_{(0,v,0)}f(x,y,h)+\sum_{k=1}^{n} f(x,y,h_1,\ldots,h_{k-1},z_k,h_{k+1},\ldots,h_n).
\end{equation}

Now $f\colon U \times (V \times  H)\rightarrow F$ is $C^{r,0}$ (see the case $s=0$). Also, $f\colon  U \times ( V \times H ) \rightarrow F $ is $C^{0,1},$  because we have just seen that  $ d^{(0,1)}f(x,(y,h),(v,z))$  exists and is given by (\ref{d-sum}), which extends continuously to $ U \times ( V \times H) \times (E_2 \times H).$

We claim that $d^{(0,1)}f\colon  U \times (( V \times H) \times (E_2 \times H))$ is $C^{r,s-1}.$ If this is true, then $f$ is $C^{r,s},$ by Lemma \ref{rn-n1}. To prove the claim, for fixed $k \in \{ 1,\ldots,n \},$ consider 
$$ \phi \colon  U \times ( V \times H \times E_2 \times H) \rightarrow F , \; (x,y,h,v,z)\mapsto f(x,y,h_1,\ldots,h_{k-1},z_k,h_{k+1},\ldots,h_n).$$ 
The map
$$ \psi \colon U \times V \times H_1 \times \cdots \times H_{n-1} \times (H_n\times E_2 \times H ) \rightarrow F,$$
$$ (x,y,h_1,\ldots,h_{n-1},(h_n,v,z)) \mapsto f(x,y,h_1,\ldots,h_n)$$
is $n$-linear in $(h_1,\ldots,h_{n-1},(h_n,v,z)).$ By induction, $\psi$ is $C^{r,s-1}$ as a map on $U \times(V\times H_1 \times \cdots \times H_{n-1} \times H_n \times E_2\times H)$. By Lemma \ref{rs-la}, also $\phi$ is $C^{r,s-1}$. Hence each of the final $k$ summands in \eqref{d-sum} is $C^{r,s-1}$ in $(x,(y,h,v,z)).$ It remains to observe that $\theta\colon U \times V \times (H\times E_2) \to F, \, (x,y,h,v) \mapsto  D_{(0,v,0)}f(x,y,h)$ is $(n+1)$-linear in the final argument and satisfies hypotheses analogous to those of $f$, with $(r,s)$ replaced by $(r,s-1).$ Hence $\theta\colon U \times V \times (H\times E_2) \to F$ is $C^{r,s-1},$ by induction. As a consequence, $d^{(0,1)}f$ is $C^{r,s-1}$ (like each of the summands in \eqref{d-sum}).           
 
\end{proof} 

Taking $E_2=\{0\},$ Lemma \ref{rs-hh} readily entails:
\begin{lem} \label{rs-hh1}
Let $r \in \N_0 \cup \{ \infty \}$, $E,H_1, \ldots, H_n, F$ locally convex spaces, $U$ be a locally convex subsets with dense interior of $E$ and  $ f\colon  U \times ( H_1 \times \cdots \times H_n) \rightarrow F $ be a $C^{r,0}$-map which is $n$-linear for fixed first argument. Then $f$ is $C^{r,\infty}.$
\end{lem}
\noindent
{\bf Proof of Lemma \ref{rn-n1}, completed.}
If $f$ is $C^{r,s}$, then $f$ is $C^{0,1}$ and $f$ is $C^{r,0}.$ Moreover $d^{(0,1)}f\colon  U \times V \times E_2 \rightarrow F$ is linear in the $E_2$-variable and 
$$ D_{(w_i,0,0)}\cdots D_{(w_1,0,0)}D_{(0,v_j,0)} \cdots D_{(0,v_1,0)}(d^{(0,1)}f)(x,y,z)$$
$$=d^{(i,j+1)}f(x,y,w_1,\ldots,w_i,z,v_1,\ldots,v_j)$$
exists for all $i,j \in \N_0$ such that $i\leq r, \; j \leq s-1,$ if $(x,y) \in U \times V,$ and extends to a continuous function in $(x,y,z,w_1,\ldots,w_i,z,v_1,\ldots,v_j) \in U \times V \times E_1 \times E_1^i\times E_2^j.$ Hence by Lemma \ref{rs-hh}, $d^{(0,1)}f$ is $ C^{r,s-1}.$

\begin{lem}  \label{lem: C11:C1}
Let $E_1$, $E_2$ and $F$ be locally convex spaces,   
$U$ and $V$ be locally convex subsets with dense interior of $E_1$ and $E_2$ respectively, and $r \in \N_0 \cup \{\infty\}$. If $f\colon  U \times V \rightarrow F$ is $C^{r,r},$ then $f$ is $C^{r}.$
\end{lem}

\begin{proof}
The proof is by the induction on $r \in \N_0 ,$ we may assume that $r< \infty.$\\
 \textit{The case $r=0$}. If $f$ is $C^{0,0},$ then $f$ is continuous and hence $C^0$.\\ 
\textit{The case $r\geq1.$} Assume $U,V$ are open subsets. Then  $D_{(w,0)}f(x,y)$ exists and is continuous in $(x,y,w),$ and $D_{(0,v)}f(x,y)$ exists and is continuous in $(x,y,v).$  Hence by \eqref{eq: rule:pdiff} $f$ is $C^1$ and 
\begin{equation} \label{eq:dwv}
df((x,y),(w,v))=D_{(w,0)}f(x,y)+D_{(0,v)}f(x,y),  
\end{equation}
which is continuous in $(x,y,w,v).$ Thus $f$ is $C^1$. In the general case, the right hand side of (\ref{eq:dwv}) is continuous for $(x,y,w,v)\in U \times V \times E_1 \times E_2$ and extends $d (f \lvert_{U^0 \times V^0})$. Hence $f$ is $C^1.$ Next, note that  
  $D_{(w,0)}f(x,y)$ and $D_{(0,v)}f(x,y)$ are $C^{r-1,r-1}$-mappings, by Lemma \ref{rn-n1} and Corollary \ref{rs-sr}. Hence $df$ is $C^{r-1}$, by induction. Since $f$ is a $C^1$ and $df$ is $C^{r-1}$, the map $f$ is $C^{r}$. 
\end{proof}

\begin{rem} \label{smoo}
If $r=\infty,$ then a map $f\colon U \times V \to F$ is $C^\infty$ if and only if it is $C^{\infty,\infty}$ (as an immediate consequence of Lemma \ref{lem: C11:C1}).              
\end{rem}

\begin{lem}[Chain Rule 1] \label{chain-1}
Let $X_1$, $X_2$, $E_1$, $E_2$ and $F$ be locally convex spaces, $P$, $Q$, $U$ and $V$ be locally convex subsets with dense interior of $X_1$, $X_2$, $E_1$ and $E_2$ respectively, $r,s \in \N_0 \cup \{\infty\}$, $f\colon  U \times V \rightarrow F $ a $C^{r,s}$-map, $g_1\colon P\rightarrow U$ a $C^r$-map and $g_2\colon Q \rightarrow V$ a $C^s$-map. Then 
$$ f \circ (g_1 \times g_2)\colon  P \times Q \rightarrow F, \; (p,q) \mapsto f(g_1(p),g_2(q))$$ 
is a $C^{r,s}$-map.
\end{lem}

\begin{proof}
 Without loss of generality, we may assume that $r,s < \infty$. The proof is by induction on $r$. 
\newline
\textit{The case $r=0$}. If $s=0$, $ f \circ (g_1 \times g_2)$ is just a composition of continuous maps, which is continuous. 
\newline
Now let $s>0.$ For fixed $x \in U$, $f_x\colon V \rightarrow F$ is $C^s.$ Hence, for fixed $p \in P$, $f_{g_1(p)}\colon V \rightarrow F$ is $C^s$ and $f_{g_1(p)} \circ g_2 \colon  Q \rightarrow F$ is $C^s$ by the Chain Rule for $C^s$-maps (see \cite{gn07}). In particular, the latter is $C^1$, whence 
$$D_{(0,z)}(f \circ(g_1 \times g_2))(p,q)= d(f_{g_1(p)} \circ g_2)(q,z)= df_{g_1(p)}(g_2(q),dg_2(q,z))$$
exists for $z \in X_2$ and $q \in Q^0.$ Hence, 
$$d^{(0,1)}(f \circ (g_1 \times g_2))(p,q,z)=\underbrace{d^{(0,1)}f}_{C^{0,s-1}}(\underbrace{g_1(p)}_{C^0 \text{ in } p},\underbrace{g_2(q),dg_2(q,z)}_{C^0 \text{ in } (q,z)})$$
exists. By induction on $s$, the map $d^{(0,1)}(f \circ (g_1 \times g_2))$ is $C^{0,s-1}$. Hence, by Lemma \ref{rn-n1}, $f \circ (g_1 \times g_2)$ is $C^{0,s}.$ \newline
\textit{Induction step $(r > 0)$}. If $s=0$, we see as in the first part of the proof that $h:=f \circ (g_1,g_2)$ is $C^{r,0}$. \newline
If $s>0,$ we know that 
$$d^{(0,1)}h(p,q,z)=\underbrace{d^{(0,1)}f}_{C^{r,s-1}}(\underbrace{g_1(p)}_{C^r},\underbrace{g_2(q), dg_2(q,z)}_{C^{s-1}}).$$
By induction on $s,$ this is $C^{r,s-1}$. Hence, by Lemma \ref{rn-n1}, $h$ is $C^{r,s}.$
\end{proof}

\begin{lem}[Chain Rule 2]\label{chain-2}
Let $E_1$, $E_2$, $F$ and $Y$ be locally convex spaces, $U$, $V$ and $W$ be locally convex subsets with dense interior of $E_1$, $E_2$ and $F$ respectively, $r,s \in \N_0 \cup \{\infty\}$, $f\colon  U \times V \rightarrow F $ a $C^{r,s}$-map with $ f(U \times V ) \subseteq W$ and $g\colon W \rightarrow Y$ be a $C^{r+s}$-map. Then 
$$ g \circ f\colon  U \times V \rightarrow Y$$ 
is a $C^{r,s}$-map.
\end{lem}

\begin{proof}
 Without loss of generality, we may assume that $r,s < \infty$. The proof is by induction on $r$. 
\newline
\textit{The case $r=0$}. If $s=0$, $g \circ f$ is just composition of continuous maps, which is continuous. 
\newline
Now let $s>0.$ For fixed $x \in U$, $f_x\colon V \rightarrow F$ is $C^s$ and $g\colon W\rightarrow Y$ is $C^s.$ Hence $g \circ f_x \colon  V \rightarrow Y$ is $C^s$ by the Chain Rule for $C^s$-maps (see \cite{gn07}). In particular, it is $C^1$, whence 
$$D_{(0,v)}(g \circ f)(x,y)=d(g \circ f_x)(y,v) =dg(f_x(y),df_x(y,v))=dg(f(x,y),d^{(0,1)}f(x,y,v))$$
exists for $v \in E_2$, if $x \in U^0, \, y \in V^0.$ Now 
$$d^{(0,1)}(g\circ f)\colon U \times (V \times E_2 ) \rightarrow Y$$
$$(x,y,v)\mapsto \underbrace{dg}_{C^{r+s-1}}(\underbrace{f(x,y),d^{(0,1)}f(x,y,v)}_{C^{0,s-1}})$$
is a $C^{0,s-1}$-map, by Lemma \ref{rn-n1} and induction on $s$. Hence, by Lemma \ref{rn-n1}, $g \circ f $ is $C^{0,s}.$ \newline
\textit{Induction step $(r > 0)$}. If $s=0$, we see as in the first part of the proof, that $h:=g \circ f$ is $C^{r,0}$. \newline
If $s>0,$ we know that $h$ is $C^{r,0}$ by the preceding. Moreover,
$$d^{(0,1)}h(x,y,v)=\underbrace{dg}_{C^{r+s-1}}(\underbrace{f(x,y),d^{(1,0)}f(x,y,v)}_{C^{r,s-1}}).$$
Hence, by induction on $s$ the map $d^{(0,1)}h$ is $C^{r,s-1}$. Hence by Lemma \ref{rn-n1}, $h$ is $C^{r,s}.$
\end{proof}

\begin{lem}[Chain Rule 3]\label{lem: chainrule}
   Let $r,s,k \in \N_0 \cup \set{\infty}$ with $k \geq r+s$. Assume that $U \subseteq E_1$ and $V \subseteq E_2, P \subseteq E_3$ are locally convex subsets with dense interior of locally convex spaces. Let $E_4$ be another locally convex space, $f \colon U \times V \rightarrow P$ of class $C^{r,s}$ and $g \colon U \times P \rightarrow E_4$ a map of class $C^{r,k}$. Then  
      $$g \circ (\pr_U , f) \colon U \times V \rightarrow E_4, (x,y) \mapsto g(x,f(x,y))$$ 
  is a map of class $C^{r,s}$, where $\pr_U \colon U \times V \rightarrow U$ is the canonical projection.
\end{lem}

\begin{proof}
 We may assume $r,s < \infty$. Use induction on $r$ starting with $r = 0$: The map $\pr_U$ is the restriction of a continuous linear map and hence $C^\infty$. We have to prove that $g \circ (\text{pr}_U , f)$ is a map of class $C^{0,s}$. Proceed by induction on $s$: As $\text{pr}_U$ is continuous $g \circ (\text{pr}_U , f)$ is continuous. We may now assume $s > 0$. Fixing $x \in U$, the Chain Rule for $C^s$-maps yields: 
      \begin{displaymath}
       d^{(0,1)} (g \circ (\text{pr}_U,f)) (x,y,z) = d^{(0,1)} g ( x, f(x,y), d^{(0,1)} f(x,y;z)) .
      \end{displaymath}
 The derivative is a composition of the $C^{0,k-1}$-map $d^{(0,1)} g$ (with respect to $U \times (P \times E_3)$) and $(\text{pr}_U , (f \circ \text{pr}_{U \times V} , d^{(0,1)} f))$. By a combination of Lemma \ref{lem: mp:prod} and Lemma \ref{chain-1}, $(f \circ \text{pr}_{U \times V} , d^{(0,1)} f) \colon U \times (V \times E_2) \rightarrow P \times E_3, (x,y,z) \mapsto (f(x,y),d^{(0,1)} f(x,y;z))$ is a map of class $C^{0,s-1}$. As $k-1 \geq s-1$, the induction hypothesis implies that $d^{(0,1)} (g \circ (\text{pr}_U , f))$ is a map of class $C^{0,s-1}$. Thus $g \circ (\text{pr}_U , f)$ is of class $C^{0,1}$ and the preceding derivative is a map of class $C^{0,s-1}$. From Lemma \ref{rn-n1} we conclude that $g \circ (\text{pr}_U , f)$ is of class $C^{0,s}$. This settles the cases $r=0, s \in \N_0$.\\ 
 We may now assume $r > 0$ and by induction the claim holds for $C^{r-1,s}$-maps for each $s \in \N_0$. We already know that $g \circ (\text{pr}_U , f)$ is a $C^{r-1,s}$-map and $g$ is a $C^1$-map by Lemma \ref{lem: C11:C1} as it is at least $C^{1,1}$. Fixing $y \in V$, the Chain Rule for $C^1$-maps and \eqref{eq: rule:pdiff} yields: 
  \begin{align}\label{eq: deriv}
   &d^{(1,0)} (g \circ (\text{pr}_U , f)) (x,y,u)\notag \\
= &d^{(1,0)} g (x,f(x,y);u) + d^{(0,1)} g(x,f(x,y); d^{(1,0)} f(x,y;u)).
  \end{align}
 The latter derivative is continuous, whence $g\circ (\text{pr}_U , f)$ is a map of class $C^{1,0}$.  
 The derivative $d^{(1,0)} g$ is of class $C^{r-1,k}$ with respect to $U \times (P \times E_1)$ which follows from Corollary \ref{rs-sr} Lemma \ref{rn-n1} and Lemma \ref{rs-hh}. Furthermore $$(f \circ \text{pr}_{U \times V} , \text{pr}_{E_1}) \colon U \times (V \times E_1) \rightarrow P \times E_1$$ is of class $C^{r,s}$, by Lemma \ref{lem: mp:prod}. Hence the induction hypothesis implies that the first summand in \eqref{eq: deriv} is of class $C^{r-1,s}$. On the other hand $d^{(0,1)} g \colon U \times (V \times E_2) \rightarrow F$ is of class $C^{r,k-1}$ by Lemma \ref{rn-n1} and by Lemma \ref{lem: mp:prod}, $(f \circ \text{pr}_{U \times V}, d^{(1,0)} f) \colon U  \times  (V \times E_1) \rightarrow P \times E_3$ is a map of class $C^{r-1,s}$ using that $d^{(0,1)}f \colon U \times (V \times E_1) \rightarrow F$ is $C^{r-1,k}$. Observe that by choice of $k$, already $k-1 \geq r-1 +s$ holds. Therefore the induction hypothesis shows that the second summand is of class $C^{r-1,s}$ with respect to $U \times (V \
times 
E_2)$. Summing up, the derivative $d^{(1,0)} (g \circ (\text{pr}_U , f))$ is a map of class $C^{r-1,s}$. We conclude from Lemma \ref{rn-n1} and Corollary \ref{rs-sr} that $g \circ (\text{pr}_U , f)$ is a map of class $C^{r,s}$.
\end{proof}

\begin{prop}  \label{eval-rk}
 Let $E$ be a finite-dimensional vector space, $F$ a locally 
 convex space, $U$ be a locally convex and locally compact subset with dense interior of $E$ and $s \in \mathbb{N}_0 \cup \{\infty\}$. Then the evaluation map 
 $$\varepsilon: C^s(U,F)\times U \rightarrow F, \; \varepsilon (\gamma,x) := \gamma(x)$$  
 of $C^{s}(U,F)$ is  $C^{\infty,s}.$ 
\end{prop}

\begin{proof}
 Without loss of generality, we may assume that $s < \infty$. The proof is by induction on $s$.
\newline
If $s=0,$ then $\varepsilon$ is continuous because $U$ is locally compact \cite[Theorem 3.4.3]{Engelking1989}. Also, $\varepsilon$ is linear in the first argument. Hence $\varepsilon$ is $C^{\infty,0},$ by Lemma \ref{rs-hh1} and Corollary \ref{rs-sr}.\newline
Let $s\geq 1.$ For $x\in U^0, \, w \in E, \, \gamma \in C^{s}(U,F)$ and small $t \in \R \setminus\{ 0 \},$
$$\frac{1}{t}(\varepsilon(\gamma,x+tw)-\varepsilon(\gamma,x)) = \frac{1}{t}( \gamma(x+tw)-\gamma(x))\to d\gamma(x,w) \text{ as }t\to 0.$$
Hence $d^{(0,1)}\varepsilon (\gamma,x,w)$ exists and is given by
\begin{equation} \label{eq:d-epsilon}
d^{(0,1)}\varepsilon(\gamma,x,w)=d\gamma(x,w)=\varepsilon_1(d\gamma,(x,w)),
\end{equation}
where $ \varepsilon_1\colon  C^{s-1}(U\times E,F)\times(U\times E)\to F , \: (\zeta,z)\mapsto \zeta(z)$
is $C^{\infty,s-1},$ by induction.\newline
The right-hand side of \eqref{eq:d-epsilon} defines a continuous map (in fact a $C^{\infty, s-1}$-map)
$$C^\infty(U,F)\times (U\times E) \to F$$
by induction and Lemma \ref{rs-la}, using that
$$C^s(U,F)\to C^{s-1}(U\times E,F), \: \gamma \mapsto d \gamma $$
is continuous linear. Thus, by Lemma \ref{rn-n1}, $\varepsilon$ is $C^{\infty,s}.$ 
\end{proof} 
 
\begin{defn} \label{Crs-top}
Let $E_1, E_2$ and $F$  be  locally convex spaces, $U$ and $V$ be locally convex subsets with dense interior of $E_1$ and $E_2$ respectively, and $r,s \in \N_0 \cup \{\infty \}.$ \newline
Give $C^{r,s}(U\times V,F)$ the initial topology with respect to the mappings
$$d^{(i,j)}\colon C^{r,s}(U \times V,F) \to C(U \times V \times E^i_1 \times E^j_2 ,F), \: \gamma \mapsto d^{(i,j)}\gamma $$
for $ \;i,j \in \N_0\text{ such that } i\leq r,j\leq s,$
where the right-hand side is equipped with the compact-open topology.
\end{defn}

\begin{lem}  \label{lem: smo-smo}
Let $E_1$, $E_2$ and $F$ be locally convex spaces,   
$U$ and $V$ be locally convex subsets with dense interior of $E_1$ and $E_2$ respectively, then 
$$C^{\infty,\infty}(U\times V,F) =C^{\infty}(U\times V,F)$$ as topological vector spaces. 
\end{lem}

\begin{proof} 
By Lemma \ref{lem: C11:C1} and Remark \ref{smoo} both spaces coincide as sets. Thus it suffices to show that the $C^{\infty,\infty}$-
topology coincides with the $C^\infty$-topology. As those topologies are initial topologies, we only have to prove that the families of maps inducing the topologies are continuous with respect to the other topology.
For $x \in U,\, y \in V,\, w:=(w_1,\ldots,w_i) \in E_1^i $ and $v:=(v_1,\ldots,v_j)\in E_2^j$, we have
$$ d^{(i,j)}f(x,y,w,v)= d^{(i+j)}f(x,y,(w_1,0),\ldots,(w_i,0),(0,v_1),\ldots,(0,v_j)).$$
Let $g\colon U\times V\times E_1^i\times E_2^j \to U\times V\times (E_1\times E_2)^{i+j},\, (x,y,w_1,\ldots,w_i,v_1,\ldots,v_j)\mapsto (x,y,(w_1,0),\ldots,(w_i,0),(0,v_1),\ldots,(0,v_j))$. As $g$ is continuous linear, by \cite[Proposition 4.4]{hg2004}, the pullback $g^*$ is continuous. Hence by continuity of $d^{(i+j)}$, $ d^{(i,j)}$ is continuous with respect to the $C^\infty$-topology. This proves that the $C^{\infty,\infty}$-topology is coarser than the $C^\infty$-topology. To show the converse we recall that $d^{(k)} f(x,y, \bullet)$ is multilinear. Writing $(w_i,v_i) = (w_i,0) + (0,v_i)$ we obtain 
  \begin{displaymath}
   d^{(k)}=\sum_{I \subseteq \set{1,\ldots,k}}g_I^* (d^{(\lvert I \rvert ,k-\lvert I \rvert )} f),
  \end{displaymath}
where we defined $g_I (x,y,(w_1,v_1),\ldots , (w_k,v_k)) \coloneq (x,y,w_{i_1}, \ldots , w_{i_{\lvert I \rvert}} , v_{j_1} , \ldots , v_{j_{k - \lvert I \rvert}})$ for $I = \set{i_1, \ldots , i_{\lvert I \rvert}}$ and $\set{1,\ldots, k} \setminus I = \set{j_1, \ldots , j_{k-\lvert I\rvert}}$. Clearly each $g_I$ is continuous linear, whence smooth and we deduce from \cite[Proposition 4.4]{hg2004} that $d^{(k)}$ is continuous with respect to the $C^{\infty,\infty}$-topology. Hence the assertion follows.
\end{proof}

\begin{lem} \label{ap-d}
Let $E$ and $F$  be  locally convex spaces, $U$ be a locally convex subset with dense interior of $E$, $r,s \in \N_0 \cup \{\infty \}.$ Then sets of the form 
$$\bigcap_{i=0}^k\{\gamma \in C^r(U,F):d^{(i)}\gamma(K_i)\subseteq Q_i\}$$
form a basis of $0$-neighbourhoods in $C^r(U,F),$ for $k\in \N_0$ such that $k \leq r,$\; compact sets $K_i \subseteq U \times E^i $ and $0$-neighbourhoods $Q_i \subseteq F.$
\end{lem}

\begin{proof}
The topology on $C^r(U,F)$ is initial with respect to the maps 
$$d^{(i)}\colon C^r(U,F)\to C(U\times E^i,F)_{c.o}, \; \gamma \mapsto d^{(i)} \gamma. $$ 
Therefore the map
$$ \Psi\colon  C^r(U,F)\to \prod_{\N_0 \ni i\le r}C(U\times E^i,F), \; \gamma \mapsto (d^{(i)} \gamma)_{\N_0 \ni i\le r}$$
is a topological embedding. Sets of the form  
$$ W:=\{(\eta_i)_{\N_0 \ni i\le r}\in \prod_{\N_0 \ni i\le r}C(U\times E^i,F):\;\eta_i(K_i)\subseteq Q_i \text{ for  }i=0,\ldots,k \}$$
(with $k\in \N_0$ such that $k \leq r,$\; compacts sets $K_i \subseteq U \times E^i $ and $0$-neighbourhoods $Q_i \subseteq F),$ form a basis of $0$-neighbourhoods in $\prod_{\N_0 \ni i\le r} C(U\times E^i,F).$ Hence the sets $\Phi^{-1}(W)$ form a basis of $0$-neighbourhoods in $C^r(U,F).$
\end{proof}

Similarly:

\begin{lem}\label{ap-f}
Let $E_1, E_2$ and $F$  be  locally convex spaces, $U$ and $V$ be locally convex subsets with dense interior of $E_1$ and $E_2$ respectively, and $r,s \in \N_0 \cup \{\infty \}.$ The sets 
$$W=\{ \gamma \in C^{r,s}(U \times V,F)\colon d^{(i,j)}\gamma(K_{i,j})\subseteq P_{i,j} \text{ for } i=0,\ldots,k \text{ and } j=0,\ldots, l \}$$
(where $k \in \N_0$ such that $k \le r$, $l \in \N_0$ such that $l\leq s,\; P_{i,j}\subseteq F$ are $0$-neighbourhoods and $K_{i,j} \subseteq U\times V\times E^i_1 \times E^j_2$ is compact) form a basis of $0$-neighbourhoods for\\ $C^{r,s}(U \times V,F).$
\end{lem}

\begin{thm}\label{vee-rs}
Let $E_1, E_2$ and $F$  be  locally convex spaces, $U$ and $V$ be locally convex subsets with dense interior of $E_1$ and $E_2$ respectively, and $r,s \in \N_0 \cup \{\infty \}.$ Then

\begin{compactenum}
\item If $\gamma \colon U \times V \to F$ is $C^{r,s},$ then $\gamma_x\colon V \to F $ is $C^s$ for all $x\in U$ and 
$$ \gamma^{\vee}\colon U \to C^s(V,F), \; x \mapsto \gamma_x$$ 
is $C^r$.
\item The map $$\Phi\colon C^{r,s}(U \times V,F)\to C^r(U,C^s(V,F)), \; \gamma \mapsto \gamma^\vee$$
is linear and a topological embedding.
\end{compactenum}

\end{thm}   

\begin{proof}
\item (a) $\gamma_x\colon V \to F $ is $C^s$ for all $x\in U$ by Lemma \ref{fx-dij}. \newline
 Since $C^{\infty}(V,F) = \varprojlim_{s \in \mathbb{N}_{0}} C^{s}(V,F)$ (\cite{gn07}), we have 
 $$C^{r}(U,C^{\infty}(V,F)) =  \varprojlim_{s \in \mathbb{N}_{0}} C^{r}(U,C^{s}(V,F)).$$ 
It therefore suffices to prove the assertion when $ s \in \mathbb{N}_{0}$ (cf. \cite[Lemma 10.3]{bgn}). We may assume that $r$ is finite. The proof is by induction on $r$. 
\newline
\textit{The case $r=0$.} If $s=0$ then the assertion follows from \cite[Theorem 3.4.1]{Engelking1989}.\newline
If $s \geq 1$, the topology on $C^{s}(V,F)$ is initial with respect to the maps
$$d^{(j)}\colon C^{s}(V,F) \to C(V \times E^j_2 ,F)_{c.o}, \: \gamma \mapsto d^{(j)}\gamma,\, \text{for } j \in \N_0\text{ such that } j\leq s.$$

Hence, we only need that $ d^{(j)} \circ f^\vee\colon  U \to C(V \times E_2^j,F)_{c.o}$ is continuous for $j \in \{0,1,\ldots,s\}.$ Now
$$ d^{(j)}(f^\vee(x))=d^{(j)}(f(x,\bullet))=d^{(0,j)}f(x,\bullet)=(d^{(0,j)}f)^\vee(x).$$
Thus $d^{(j)} \circ f^\vee = (d^{(0,j)}f)^\vee \colon  U \to C(V \times E_2^j,F)_{c.o},$ which is continuous by induction. As a consequence, $ \gamma^\vee \colon U \to C^s(V,F)$ is continuous. \newline
\textit{The case $r\geq 1$}. If $s=0.$ then $f^\vee\colon U \to C(V,F).$
Let $x \in U^0, \; z \in E_1.$ Then $x+tz \in U^0,$ for small $t \in \R \cup \{\infty\};$ we show that
$$\frac{1}{t}(f^\vee(x+tz)-f^\vee(x)) \to d^{(1,0)}f(x,\bullet,z) $$ in $C(V,F)$ as $t\to 0.$
For this, let $K \subseteq V $ be compact. we have to show that  
$$ (\frac{1}{t}(f^\vee(x+tz)-f^\vee(x)))\lvert_{K} \to (d^{(1,0)}f(x,\bullet,z))\lvert_{K} $$ uniformly  as $t\to 0$. Let $W \subseteq F$ be a $0$-neighbourhood. Without loss of generality, $W$ is closed and absolutely convex. There is $\varepsilon \geq 0$ such that $x+B^\R_\varepsilon(0)z\subseteq U^0.$ For $y \in K$ and $t \in \R \setminus  \{0\} $ such that $\left| t \right|  < \varepsilon,$ we have
\begin{align*}
\Delta(t,y):&= \frac{1}{t}(f^\vee(x+tz)-f^\vee(x))(y)- d^{(1,0)}f(x,y,z) \\
            &= \frac{1}{t}(f(x+tz,y)-f(x,y))- d^{(1,0)}f(x,y,z) \\  
            &= \int^1_0  d^{(1,0)}f(x+\sigma tz,y,z)d\sigma - d^{(1,0)}f(x,y,z)\\
            &= \int^1_0 ( d^{(1,0)}f(x+\sigma tz,y,z) - d^{(1,0)}f(x,y,z))\, d\sigma.
\end{align*}
The function 
$$g\colon B^\R_\varepsilon(0)\times K\times [0,1] \to F, (t,y,\sigma)\longmapsto d^{(1,0)}f(x+\sigma tz,y,z)-d^{(1,0)}f(x,y,z)$$
is continuous and $g(0,y,\sigma)=0$ for all $(y,\sigma) \in K \times [0,1].$ Because $K \times [0,1]$ is compact, by the Wallace Lemma (see \cite[3.2.10]{Engelking1989}), there exists $\delta \in (0,\varepsilon] $ such that $ g(B^\R_\delta (0) \times K \times [0,1])\subseteq W.$ Hence 
$\Delta(t,y)= \int^1_0 g(t,y,\sigma)d \sigma \in W$ for all $ y \in K$ and all $t \in B^\R_\delta(0) \setminus \{0\}.$ Because this holds for all $y \in K,$ we see that $ \Delta(t,\bullet) \to 0$ uniformly, as required. Thus $df^\vee (x,z)$ exists for all $x \in U^0, \; z \in E_1$ and is given by $df^\vee(x,z)=d^{(1,0)}f(x,\bullet,z).$ Now
$$U \to C(V,F), x \mapsto d^{(1,0)}f(x,\bullet,z)$$
is a continuous function in all of $U$ (by $r=0$); so $f^\vee$ is $C^1$ on $U,$ and $ df^\vee(x,z)=d^{(1,0)}f(x,\bullet,z).$ Because 
$$h\colon (U\times E_1)\times V \to F, \; ((x,z),y) \mapsto d^{(1,0)}f(x,y,z)$$
is $C^{(r-1,0)}$ (see Lemma \ref{rn-n1} and Corollary \ref{rs-sr}), by induction $d(f^\vee)=h^\vee\colon U \times E_1 \to C(V,F)$ is $C^{r-1}$. Hence $f$ is $C^r.$ \newline
Let $s \geq 1.$ Because 
$$C^s(V,F)\to C(V,F)\times C^{s-1}(V \times E_2,F), \; \gamma \mapsto (\gamma,d\gamma)$$
is a linear topological embedding with closed image (see \cite[Lemma 91]{al}), $f^\vee\colon U\to C^s (V,F)$ will be $C^r$ if $f^\vee\colon U\to C(V,F)$ is $C^r$ (which holds by induction) and the map 
$$ h\colon U\to C^{s-1}(U\times E_2,F), \; x \mapsto d(f^\vee(x))$$
is $C^r$ (see \cite{gn07}; cf. \cite[Lemma 10.1]{bgn}) For $x \in U, \; y \in V$ and $z \in E_2,$ we have 
$$ h(x)(y,z)= d(f^\vee(x))(y,z)=d(f(x,\bullet))(y,z)=d^{(0,1)}f(x,y,z),$$
thus $h=(d^{(0,1)}f)^\vee$ for $d^{(0,1)}f\colon U \times (V\times E_2)\to F.$ This function is $C^{r,s-1}$ by Lemma \ref{rn-n1}. Hence $h$ is $C^r$ by induction.
\item (b) \rm{The linearity of $\Phi$} is clear. For $y \in V,$ the point evaluation $\lambda\colon C^s(V,F)\to F,\; \eta \mapsto \eta (y)$   
is continuous linear. Hence        
\begin{align*}
(d^{(i)}f^\vee)(x,w_1,\ldots, w_i)(y)&=\lambda ((d^{(i)}f^\vee)(x,w_1,\ldots,w_i))\\
                                     &=d^{(i)}(\lambda \circ f^\vee)(x,w_1,\ldots,w_i)\\
                                     &=d^{(i)}(f(\bullet,y)(x,w_1,\ldots,w_i))\\                                                          
                                     &=d^{(i,0)}f(x,y,w_1,\ldots,w_i),
\end{align*}                                          
using that $(\lambda \circ f^\vee)(x)=\lambda(f^\vee(x))=f^\vee(x)(y)=f(x,y).$
Hence 
$$(d^{(i)}f^\vee)(x,w_1,\ldots,w_i)=(d^{(i,0)}f)(x,\bullet,w_1,\ldots,w_i).$$
Hence by Schwarz' Theorem (Theorem \ref{schwarz}) $$d^{(j)}((d^{(i)}f^\vee)(x,w_1,\ldots,w_i))(y,v_1,\ldots,v_j)=d^{(i,j)}f(x,y,w_1,\ldots,w_i,v_1,\ldots,v_j).$$\newline
\textit{$\Phi$ is continuous at 0.} Let $W \subseteq C^r(U,C^s(V,F))$ be a $0$-neighbourhood. After shrinking $W,$ without loss of generality 
$$W=\bigcap_{i=0}^k\{\gamma \in C^r(U,C^s(V,F))\colon d^{(i)}\gamma(K_i)\subseteq Q_i\}$$
where $k \in \N_0$ with $k \leq r,\, K_i \subseteq U \times E_1^i$ is compact and $Q_i \subseteq C^s(V,F)$ is a $0$-neighbourhood (see Lemma \ref{ap-d}). Using Lemma \ref{ap-d} again, after shrinking  $Q_i$ we may assume that 
$$Q_i= \bigcap_{j=0}^{l_i}\{\eta \in C^s(V,F)\colon d^{(j)}\eta(L_{i,j})\subseteq P_{i,j}\}$$
with $l_i \in \N_0,$ such that $l_i \le s,$ compact sets $ L_{i,j} \subseteq V \times E^j_2$ and $0$-neighbourhoods $P_{i,j} \subseteq F$ shrinking  $Q_i$ further, we may assume that $l_i=l$ is independent of $i.$ Then $W$ is the set of all 
$ \gamma \in C^r(U,C^s(V,F))$ such that $d^{(j)}(d^{(i)}\gamma(x,w))(y,v)\in P_{i,j}$
for all $ i=0,\ldots,k \text{ and } j=0,\ldots, l, \; (x,w) \in K_i \subseteq U \times E^i_1$ and $(y,v)\in L_{i,j} \subseteq V \times E_2^j .$ The projections of $U \times E_1^i$ onto the factors $U$ and $E^i_1$ are continuous, hence the images $K^1_i$ and $K^2_i$ of $K_i$ under these projections are compact. After replacing $K_i$ by $K_i^1 \times K_i^2,$ without loss of generality $K_i=K_i^1 \times K_i^2.$ Likewise, without loss of generality $L_{i,j}= L_{i,j}^1 \times  L_{i,j}^2 $ with compact sets  $L_{i,j}^1 \subseteq V$ and $L_{i,j}^2 \subseteq  E^j_2.$\newline
Now if $\gamma \in C^{r,s}(U \times V,F)$ then $d^{(j)}(d^{(i)}\gamma^\vee (x,w))(y,v)=d^{(i,j)}\gamma(x,y,w,v).$ Hence $\gamma^\vee \in W$ if and only if $d^{(i,j)}\gamma(K_i^1\times L_{i,j}^1 \times K_i^2 \times L_{i,j}^2) \subseteq P_{i,j}$ for all $i=0,\ldots,k$ and $j=0,\ldots, l_i.$ This is a basic neighbourhood in $C^{r,s}(U\times V, F)$ (see Lemma \ref{ap-f}). Thus $\Phi^{-1}(W)$ is a $0$-neighbourhood, whence $\Phi$ is continuous at $0$, and hence $\Phi$ is continuous.\newline
It is clear that $\Phi$ is injective. To see that $\Phi$ is an embedding, it remains to show that $\Phi(W)$ is a $0$-neighbourhood in $\im(\Phi)$ for each $W$ in a basis of $0$-neighbourhoods in $C^{r,s}(U\times V,F).$\newline
Take $W$ as in Lemma \ref{ap-f}; without loss of generality, after increasing $K_{i,j},$ we may assume $K_{i,j}=K_{i,j}^1 \times L_{i,j}^1 \times K_{i,j}^2\times L_{i,j}^2$ with compact sets $K_{i,j}^1 \subseteq U,\; L_{i,j}^1 \subseteq V, \;  K_{i,j}^2\subseteq E_1^i$ and $L_{i,j}^2 \subseteq E_2^j.$ Then $\Phi(W):=\{\eta \in \im(\Phi)\colon  d^{(j)}(d^{(i)}\eta(x,w))(y,v) \in P_{i,j}\}$ 
for all $ i=0,\ldots,k, \;j=0,\ldots,l,\; x\in K_{i,j}^1 ,\;y\in L_{i,j}^1, \;  w \in K_{i,j}^2$ and $ v \in L_{i,j}^2 ,$ 
which is a $0$-neighbourhood in $\im (\Phi),$ by Lemma \ref{ap-d}.       
\end{proof}

\begin{lem} \label{ek-e}
Let $X$ be a topological space, $E$ and $F$ be locally convex spaces, $k \in \N,$ and $ f\colon X \times E^k \to F$ be a map such that $f(x,\bullet)\colon E^k \to F $ is symmetric $k$-linear for each $x \in X.$ Then $f$ is continuous if and only if $g\colon X\times E \to F, \; (x,w) \mapsto f(x,w,\ldots,w)$ is continuous.
\end{lem} 

\begin{proof}
The continuity of  $g$ follows directly from the continuity of $f$. If, conversely, $g$ is continuous, then by the Polarization Identity \cite[Theorem A]{BoSi1971}
$$f(x,w_1,\ldots,w_k)=\frac{1}{k!} \sum_{\varepsilon_{1},\ldots,\varepsilon_{k}=0}^1 (-1)^{k-(\varepsilon_1+\cdots + \varepsilon_k)}g(x,\varepsilon_1 w_1+\cdots+\varepsilon_k w_k),$$
which is continuous.

\end{proof}  

\begin{lem} \label{kl-ee}
Let $X$ be a topological space, $E_1,\; E_2$ and $F$ be locally convex spaces, $k,l \in \N,$ and $f\colon X\times E_1^k \times E_2^l \to F$ be a map such that $f(x,\bullet,w_1,\ldots,w_l)\colon E_1^k\to F$ is symmetric $k$-linear for all $x \in X$ and $w_1,\ldots,w_l \in E_2,$ and $f(x,v_1,\ldots,v_k,\bullet)\colon E_2^l \to F$ is symmetric $l$-linear for all $x \in X$ and $v_1,\ldots,v_k\in E_1.$ Then $f$ is continuous if and only if $g\colon X \times E_1 \times E_2 \to F, \; g(x,v,w):=f(x,v,\ldots,v,w,\ldots,w)$ is continuous.
\end{lem}

\begin{proof}
The continuity of  $g$ follows directly from the continuity of $f$. If, conversely, $g$ is continuous, then two applications of the Polarization Identity show that
\begin{align*}
& f( x,v_1,\ldots,v_k,w_1,\ldots,w_l)\\
&=\frac{1}{l!} \sum_{\varepsilon_{1},\ldots,\varepsilon_{l}=0}^1 (-1)^{l-(\varepsilon_1+\cdots + \varepsilon_l)}f(x,v_1,\ldots,v_k, \sum_{j=1}^l \varepsilon_j w_j,\ldots, \sum_{j=1}^l \varepsilon_j w_j)\\
&= \frac{1}{k!\,l!} \sum_{\varepsilon_{1},\ldots,\varepsilon_{l},\delta_1,\ldots,\delta_k=0}^1 (-1)^{l-(\varepsilon_1+\cdots + \varepsilon_l)} (-1)^{k-(\delta_1+\cdots + \delta_k)}g(x,\sum_{i=1}^k \delta_i v_i, \sum_{j=1}^l \varepsilon_j w_j),
\end{align*}                                  
whence $f$ is continuous.

\end{proof}

\begin{thm} {\bf(Exponential Law).} \label{kspace-iso}
Let $E_1, E_2$ and $F$  be  locally convex spaces, $U$ and $V$ be locally convex subsets with dense interior of $E_1$ and $E_2$ respectively, and $r,s \in \N_0 \cup \{\infty \}.$  Assume that at least one of the following conditions is satisfied:
\begin{compactenum}
\item $V$ is locally compact.
\item $r=s=0$ and $U \times V$ is a $k$-space.
\item $r\geq 1, \; s=0$ and $U \times V \times E_1 $ is a $k$-space.
\item $r=0, \; s\geq 1$ and $U \times V \times E_2 $ is a $k$-space.
\item $r\geq 1, \; s\geq1$ and $U \times V \times E_1 \times E_2 $ is a $k$-space.
\end{compactenum}
Then 
$$\Phi \colon C^{r,s}(U\times V,F)\to C^r(U,C^s(V,F)), \; f \mapsto f^\vee$$
is an isomorphism of topological vector spaces. Moreover, if $g\colon U \to C^s(V,F)$ is $C^r,$ then 
$$g^\wedge \colon U \times V \to F,\; g^\wedge (x,y):= g(x)(y)$$
is $C^{r,s}.$
\end{thm}

\begin{proof}
By Theorem \ref{vee-rs}, we only need to show the final assertion. In fact, given $g \in C^r(U,C^s(V,F)),$ the map $g^{\wedge}$ will be $C^{r,s}$ then, and hence $g=({g^\wedge})^\vee =\Phi(g^{\wedge}).$ Thus $\Phi $ will be surjective. Hence by Theorem \ref{vee-rs}, $\Phi$ will be an isomorphism of topological vector spaces.

\item (a) ${g}^\wedge(x,y) = g(x)(y) = \varepsilon (g(x),y)$ where $\varepsilon \colon C^s(V,F) \times V \to F, \; (\gamma,y)\mapsto \gamma (y)$ is $C^{\infty,s}$ (Proposition \ref{eval-rk}). Hence $ {g}^\wedge$ is $C^{r,s}$ by Chain Rule 1 (Lemma \ref{chain-1}). \newline
\textit{$k$-space conditions}. If $g\colon U \to C^s(V,F)$ is $C^r$, define $g^\wedge\colon U \times V \to F, \;g^\wedge(x,y)=g(x)(y).$ For fixed $x \in U,$ we have $g^\wedge(x,\bullet)=g(x)$ which is $C^s,$ hence 
\begin{align*}
(D_{(0,v_j)} \cdots D_{(0,v_1)} g^\wedge )(x,y)&= d^{(j)} (g(x))(y,v_1,\ldots, v_j)\\
                                               &= (d^{(j)}\circ g) (x)(y,v_1,\ldots,v_j)
\end{align*}
exists for  $j \in \N_0\text{ such that } j\leq s,\; y \in V^0$ and $v_1,\ldots,v_j \in E_2.$ Also,
$$ (D_{(0,v_j)} \cdots D_{(0,v_1)} g^\wedge )(x,y) = (\varepsilon _{(y,v_1,\ldots,v_j)}\circ d^{(j)}\circ g )(x),$$
where $\varepsilon _{(y,v_1,\ldots,v_j)} \colon  C^{s-j}(V\times E_2^j,F)\to F ,\; f\mapsto f(y,v_1,\ldots,v_j).$ For fixed $(y,v_1,\ldots,v_j),$ this is the function $ \varepsilon _{(y,v_1,\ldots,v_j)}\circ d^{(j)}\circ g $ of $x,$ which is $C^r.$  Since $\varepsilon _{(y,v_1,\ldots,v_j)}$ and $d^{(j)}\colon C^s(V,F) \to C^{s-j}(V \times E^j_2,F)$ are continuous linear, we obtain the directional derivatives    
\begin{align*}
&(D_{(w_i,0)}\cdots D_{(w_1,0)}D_{(0,v_j)} \cdots D_{(0,v_1)} g)(x,y)\\ 
&=\varepsilon _{(y,v_1,\ldots,v_j)}( d^{(j)}(d^{(i)} g(x,w_1,\ldots, w_i)))\\
&= d^{(j)}(d^{(i)} g(x,w_1,\ldots, w_i))( y,v_1,\ldots,v_j)\\
&=(d^{(j)}\circ(d^{(i)} g))(x,w_1,\ldots, w_i)( y,v_1,\ldots,v_j)\\
&=(d^{(j)}\circ(d^{(i)} g))^\wedge((x,w_1,\ldots, w_i),( y,v_1,\ldots,v_j))
\end{align*}  
for $x \in U^0,\; w_1,\ldots, w_i \in E_1,$ and $ i \in \N_0$ such that $i \leq r.$ 
To see that $g^\wedge $ is $C^{r,s},$ it therefore suffices to show that 
$$h\colon ( d^{(j)}\circ(d^{(i)} g))^\wedge\colon  U\times E^i_1 \times V\times E^j_2 \to F$$
is continuous for all $i,j \in \N_0$ such that $i \leq r, j \leq s.$\newline
\textit{The case $i=0,\; j=0.$} Then $h=g^\wedge,$ which is continuous by the case of topological spaces with $U \times V $ a $k$-space (see \cite[Proposition B.15]{Glox2}).\newline
\textit{The case $i=0,\; j\geq 1.$} Then $$h\colon (U\times V) \times E^j_2 \to F,\, h(x,y,\bullet):=d^{(j)}(g(x))(y,\bullet)\colon E^j_2 \to F $$ is symmetric $j$-linear. Hence, by Lemma \ref{ek-e}, $h$ is continuous if we can show that $ f\colon U\times V\times E_2\to F,\; (x,y,v)\mapsto d^{(j)}(g(x))(y,v,\ldots,v)=h(x,y,v,\ldots,v)$ is continuous.\newline
Now
$$
\begin{CD}
C^s (V,F) @> d^{(j)} >> C^0(V\times E^j_2,F) \\
@AA g A @VV C^0(\varphi,F)  V\\
U @> \eta >>C^0(V\times E_2, F).
\end{CD}
$$ 
\newline   
where $\varphi \colon V\times E_2 \to V \times E_2^j, \; (y,v)\mapsto (y,v,\ldots,v)$ and $C^0(\varphi,F)\colon C^0(V\times E^j_2,F) \to C^0(V \times E_2,F), \; \gamma \mapsto \gamma \circ \varphi$ is the pullback which is continuous linear (see \cite{gn07}; cf. \cite[Lemma 4.4]{hg2004}).

Hence $\eta\coloneq C^0(\varphi,F) \circ d^{(j)} \circ g\colon  U \to C^0(V\times E_2,F)$ is continuous. Because\\
 $U \times (V \times E_2)$ is a $k$-space by hypothesis, we know from the case of topological spaces (see \cite[Proposition B.15]{Glox2}) that $ f=\eta^\wedge\colon U \times (V\times E_2)\to F$ is continuous. \newline
\textit{The case $i \geq 1, \;j =0$.} Then 
$$h\colon U\times E^i_1 \times V \to F,\; h(x,w_1, \ldots ,w_j,y)=(d^{(j)}g)(x,w_1,\ldots , w_j)(y).$$ 
By Lemma \ref{ek-e}, $h$ is continuous if we can show that $f \colon U \times E_1 \times V \to F, \; f(x,w,y):=(d^{(i)}g)(x,w,\ldots,w)(y)$ is continuous. But $f=\psi^\wedge$ for the continuous map $\psi\colon  U \times E_1 \to C^0(V,F), \; (x,w)\mapsto (d^{(i)}g)(x,w,\ldots,w).$ Hence $f$ is continuous because $U \times E_1 \times V$ is a $k$-space by hypothesis.\newline
\textit{The case $i \geq 1,\;j \geq  1.$}  By Lemma \ref{kl-ee}, $h$ will be continuous if we can show that 
$$ f\colon U\times E_1 \times V \times E_2 \to F, \; f(x,w,y,v):=h(x,\underbrace{w,\ldots,w}_{i-\text{times}}, y,\underbrace{v,\ldots,v}_{j-\text{times}})$$
 is continuous. Now $\psi \colon  U \times E_1 \to U \times E^i_1, \; (x,w) \mapsto (x,w,\ldots,w)$ is continuous and $\theta:= C^0(\varphi,F) \circ d^{(j)}\circ d^{(i)}g \circ \psi\colon  U \times E_1 \to C^0(V\times E_2 ,F)$ is continuous. Since $ U \times E_1 \times V \times E_2$ is a $k$-space by hypothesis, it follows that $\theta^\wedge\colon U\times E_1 \times V\times E_2 \to F$ is continuous (see \cite[Proposition B.15]{Glox2}). But $\theta^\wedge=f,$ and thus $f$ is continuous.  
\end{proof}

\section{The Exponential Law for mappings on manifolds}

\begin{defn}\label{defnNEW}
We recall from \cite{gn07} that a \emph{manifold with rough boundary} modelled on a locally convex space~$E$ is a Hausdorff topological space $M$ with an atlas of smoothly compatible
homeomorphisms $\phi\colon U_\phi\to V_\phi$ from open subsets $U_\phi$ of $M$ onto locally convex subsets $V_\phi\subseteq E$ with dense interior.
If each $V_\phi$ is open, $M$ is an ordinary manifold (without boundary). If each $V_\phi$ is relatively open in a closed hyperplane $\lambda^{-1}([0,\infty[)$, where $\lambda\in E'$ (the space of continuous linear functional on $E$), then $M$ is a \emph{manifold with smooth boundary}. In the case of a \emph{manifold with corners}, each $V_\phi$ is a relatively open
subset of $\lambda^{-1}_1([0,\infty[)\cap\cdots\cap \lambda^{-1}_n([0,\infty[)$, for suitable $n\in \N$ (which may depend on $\phi$) and linearly independent $\lambda_1,\ldots,\lambda_n\in E'$.
\end{defn}

\begin{defn} \label{ch-rs}
Let $M_1$ and $M_2$ be smooth manifolds (possibly with rough boundary) modelled on locally convex spaces, $r,s \in \N_0 \cup \{\infty \}$ and $F$ be a locally convex space. A map $f\colon  M_1 \times M_2 \to F$ is called $C^{r,s}$ if $f \circ (\varphi^{-1} \times \psi^{-1})\colon V_\varphi \times V_\psi \to F$ is $C^{r,s}$ for all charts $ \varphi\colon U_\varphi \to V_\varphi $ of $M_1$ and $\psi\colon U_\psi \to V_\psi $ of $ M_2.$ Then $f$ is continuous in particular.
\end{defn}

\begin{defn} \label{ch-rs-in}
In the situation of  Definition \ref{ch-rs}, let $C^{r,s}(M_1 \times M_2,F)$ be the space of all $C^{r,s}$-maps $f\colon M_1 \times M_2 \to F.$ Endow $C^{r,s}(M_1 \times M_2,F)$  with the initial topology with respect to the maps $C^{r,s}(M_1 \times M_2,F) \to C^{r,s}(V_\varphi \times V_\psi,F), \; f \mapsto f \circ (\varphi^{-1} \times \psi^{-1}),$ for $\varphi$ and $\psi$ in the maximal smooth atlas of $M_1$ and $M_2$, respectively.    
\end{defn}

The following fact is well known (cf. \cite[Proposition 2.3.2]{Engelking1989}). 
\begin{lem} \label{prod-emb}
Let $(\theta_j)_{j \in J}$ be a family of topological embeddings $ \theta_j\colon X_j \to Y_j$ between topological spaces. Then also
$$ \theta \coloneq \prod_{j \in J} \theta_j \colon  \prod_{j \in J} X_j \to \prod_{j \in J} Y_j, \;(x_j)_{j \in J} \mapsto (\theta_j(x_j))_{j \in J}$$
is a topological embedding.
\end{lem}

\begin{prop}\label{v-mfd-rs}
Let $M_1$ and $M_2$ be smooth manifolds (possibly with rough boundary) modelled on locally convex spaces, $r,s \in \N_0 \cup \{\infty \}$ and $F$ be a locally convex space. Then
\begin{compactenum}
 \item  $f^\vee \in C^{r}(M_1,C^s(M_2,F))$  for all $f \in C^{r,s}(M_1 \times M_2,F)$. 
 \item  The map
$$\Phi\colon C^{r,s}(M_1 \times M_2,F)\to C^r(M_1,C^s(M_2,F)), \; f \mapsto f^\vee $$
is linear and a topological embedding.
\end{compactenum}
\end{prop}

\begin{proof} (a) It is clear that $f^\vee(x)=f(x,\bullet)$ is a $C^s$-map $M_2\to F.$ It suffices to show that $f\circ \varphi^{-1}\colon U_\varphi \to C^s(M_2,F)$ is $C^r$ for each chart $\varphi\colon U_\varphi \to V_\varphi$ of $M_1.$ Let $\mathcal{A}_2$ be the maximal smooth atlas for $ M_2.$ Because the map $$\Psi \colon C^s(M_2,F) \to \prod_{\psi \in \mathcal{A}_2} C^s(U_\psi,F),\; h \mapsto (h \circ \psi^{-1})_{\psi \in \mathcal{A}_2}$$
is a linear topological embedding with closed image (see \cite{gn07}; cf.\ \cite[4.7 and Proposition 4.19(d)]{hg2004}). $f^\vee \circ \varphi^{-1}$ is $C^r$ if and only if $\Psi \circ f^\vee \circ \varphi^{-1}$ is $C^r$ (see \cite{gn07}; cf. \cite[Lemma 10.2]{bgn}), which holds if all components are $C^r$. Hence we only need that 
$$\theta\colon V_\varphi \to C^s(V_\psi,F), \; x \mapsto f^\vee (\varphi^{-1}(x)) \circ \psi^{-1}=(f \circ (\varphi^{-1} \times \psi^{-1}))^\vee (x)$$
is $C^r.$ But $\theta=(f \circ (\varphi^{-1} \times \psi^{-1}))^\vee$ where $f \circ (\varphi^{-1} \times \psi^{-1})\colon V_\varphi \times V_\psi \to F$ is $C^{r,s},$ hence $\theta$ is $C^r$ by Theorem \ref{vee-rs}.\newline
(b) It is clear that $\Phi $ is linear and injective. Because $\Psi$ is linear and a topological embedding, also 
$$C^r(M_1,\Psi)\colon C^r(M_1,C^s(M_2,F))\to C^r(M_1, \prod_{\psi \in \mathcal{A}_2}C^s (V_\psi,F)),\; f \mapsto \Psi \circ f$$
is a topological embedding \cite{gn07}.\newline
Let $P:=\prod_{\psi \in \mathcal{A}_2}C^s (V_\psi,F).$ The map 
$$\Xi\colon C^r (M_1,P) \to \prod_{\varphi \in \mathcal{A}_1}C^r(V_\varphi,P), \; f \mapsto (f \circ \varphi^{-1})_{\varphi \in \mathcal{A}_1}$$
is a linear topological embedding. Using the isomorphism $$\prod_{\varphi \in \mathcal{A}_1}C^r(V_\varphi,P) \cong \prod_{\varphi \in \mathcal{A}_1}\prod_{\psi \in \mathcal{A}_2}  C^r(V_\varphi,C^s (V_\psi,F))$$ we obtain a linear topological embedding 
\begin{align*}
 \Gamma:=\Xi \circ C^r(M_1,\Psi)\colon C^r(M_1,C^s(M_2,F)) \to  \prod_{\varphi \in \mathcal{A}_1}\prod_{\psi \in \mathcal{A}_2}  C^r(V_\varphi,C^s (V_\psi,F)),\\
	    f \mapsto (C^s(\psi^{-1},F) \circ f \circ \varphi^{-1})_{\substack {\varphi \in \mathcal{A}_1,\\ \psi \in \mathcal{A}_2 }} \quad \quad \quad \quad \quad \quad \quad \quad
\end{align*}
where $C^s(\psi^{-1},F)\colon  C^s(M_2,F) \mapsto C^s(V_\psi,F), \; f \mapsto f \circ \psi^{-1}.$ Also the map
$$\omega\colon C^{r,s} (M_1 \times M_2,F)\to \prod_{\substack {\varphi \in \mathcal{A}_1,\\ \psi \in \mathcal{A}_2 }}C^{r,s}(V_\varphi \times V_\psi,F),\, f \mapsto (f \circ (\psi^{-1} \times \varphi^{-1}))_{\substack {\varphi \in \mathcal{A}_1,\\ \psi \in \mathcal{A}_2 }}$$
is a topological embedding, by Definition \ref{ch-rs-in}. Now we have the commutative diagram.

$$
\begin{CD}
C^{r,s} (M_1 \times M_2,F) @> \Phi >> C^r (M_1,C^s(M_2,F)) \\
@VV \omega V @VV \Gamma  V\\
{\displaystyle \prod_{\substack {\varphi \in \mathcal{A}_1,\\ \psi \in \mathcal{A}_2 }}}C^{r,s}(V_\varphi \times V_\psi,F) @> \eta >>{\displaystyle \prod_{\substack {\varphi \in \mathcal{A}_1,\\ \psi \in \mathcal{A}_2 }}}C^{r}(V_\varphi, C^s( V_\psi,F))
\end{CD}
$$ 
\newline   
where $\eta$  is the map $(f_{\varphi,\psi})_{\varphi \in \mathcal{A}_1, \psi \in \mathcal{A}_2} \mapsto (f^\vee_{\varphi,\psi})_{\varphi \in \mathcal{A}_1, \psi \in \mathcal{A}_2 }$.
Because the vertical arrows are topological embeddings and also the horizontal  arrow at the bottom (by Lemma \ref{prod-emb} and Theorem \ref{vee-rs}) is a topological embedding, we deduce that the map $\Phi$ at the top has to be a topological embedding as well. Here, we used that open subsets of $k$-spaces are $k$-spaces.
\end{proof}

\begin{thm} \label{k-mfd-iso}
Let $M_1$ and $M_2$ be smooth manifolds (possibly with rough boundary) modelled on locally convex spaces $E_1$ and $E_2$ respectively, $F$ be a locally convex space and $r,s \in \N_0 \cup \{\infty \}$. Assume that $M_2$ is locally compact or that one of the following conditions is satisfied:
\begin{compactenum}
\item $r=s=0$ and $M_1 \times M_2$ is a $k$-space.
\item $r\geq 1, \; s=0$ and $M_1 \times M_2 \times E_1 $ is a $k$-space.
\item $r=0, \; s\geq 1$ and $M_1 \times M_2 \times E_2 $ is a $k$-space.
\item $r\geq 1, \; s\geq1$ and $M_1 \times M_2 \times E_1 \times E_2 $ is a $k$-space.
\end{compactenum}
Then 
$$\Phi \colon C^{r,s}(M_1\times M_2,F)\to C^r(M_1,C^s(M_2,F)), \; f \mapsto f^\vee$$
is an isomorphism of topological vector spaces. Moreover, a map $g\colon M_1 \to C^s(M_2,F)$ is $C^r$ if and only if
$$g^\wedge \colon M_1 \times M_2 \to F,\; g^\wedge (x,y):= g(x)(y)$$
is $C^{r,s}.$
\end{thm}

\begin{proof}
By Proposition \ref{v-mfd-rs}, we only need to show that $\Phi$ is surjective. To this end, Let $g \in  C^r(M_1,C^s(M_2,F))$ and define 
$$ f\coloneq g^\wedge \colon M_1 \times M_2 \to F,\; f (x,y)\coloneq g(x)(y).$$
Let $ \varphi$ and $ \psi$ be charts for $M_1$ and $M_2,$ respectively. Then 
$$f \circ (\varphi^{-1} \times \psi^{-1})\colon V_\varphi \times V_\psi \to F, \; (x,y) \mapsto (C^s(\psi^{-1},F)\circ g\circ \varphi^{-1})^\wedge (x,y)$$
with $C^s(\psi^{-1},F)\colon  C^s(M_2,F) \to  C^s(V_\psi,F), h \mapsto  h \circ \psi^{-1} $ continuous linear. Hence $C^s(\psi^{-1},F)\circ g \circ \varphi^{-1}\colon V_\varphi \to  C^s( V_\psi,F)$ is $C^r.$ Hence $f \circ (\varphi^{-1} \times \psi^{-1})$ is $C^{r,s}$  by the exponential law (Theorem \ref{kspace-iso}). \newline
\textit{Note.} In (d) $ V_\varphi \times V_\psi \times E_1 \times E_2$ is homeomorphic to the open subset  $U_\varphi \times U_\psi \times E_1\times E_2$ of the $k$-space $ M_1 \times M_2 \times E_1    \times E_2$ and hence a $k$-space. Similarly in (a), (b) and (c). Hence the Exponential Law (Theorem \ref{kspace-iso}) applies. If $M_2$ is locally compact, then the open subsets the $U_\psi$ are locally compact and hence also the $V_\psi.$ Again, the Exponential Law (Theorem \ref{kspace-iso}) applies.
\end{proof} 

To deduce a corollary, we use the following lemma.

\begin{lem} \label{cap-ksp}
Let $X$ be a Hausdorff topological space. If $X=\bigcup_{j \in J} V_j $ with open subsets $V_j \subseteq X$ which are $k$-spaces, then $X$ is a $k$-space.
\end{lem}

\begin{proof}
Let $W \subseteq X$ be a subset such that $W \cap K$ is relatively open in $K$ for each compact subset $K \subseteq X.$ We show that $W$ is open in $X.$ Since $W=\bigcup_{j \in J}(W \cap V_j),$ it suffices to show that each $V_j \cap W$ is open. For each compact subset $ K \subseteq V_j,\; K \cap(V_j \cap W)=K \cap W$ is relatively open  in $K$ by hypothesis, thus $V_j \cap W$ is open in $V_j,$ hence open in $X.$
\end{proof}

\begin{cor} \label{mm-iso}
Let $M_1$ and $M_2$ be smooth manifolds (possibly with rough boundary) modelled on locally convex spaces $E_1$ and $E_2$ respectively, $F$ be a locally convex space and $r,s \in \N_0 \cup \{\infty \}$. Assume that {\rm (a)}, {\rm (b)} or {\rm (c)} is satisfied:
\begin{compactenum}
 \item $M_2$ is a finite-dimensional manifold with corners.\footnote{In this case we have no further assumptions on the space $E_1$.}
 \item $E_1$ and $E_2$ are metrizable.
 \item $M_1$ and $M_2$ are manifolds with corners and both of $E_1$ and $E_2$ are hemicompact $k$-spaces.
\end{compactenum}
Then
$$\Phi \colon C^{r,s}(M_1\times M_2,F)\to C^r(M_1,C^s(M_2,F)), \; f \mapsto f^\vee$$
is an isomorphism of  topological vector spaces. Moreover, a map $g\colon M_1 \to C^s(M_2,F)$ is $C^r$ if and only if
$$g^\wedge \colon M_1 \times M_2 \to F,\; g^\wedge (x,y)\coloneq g(x)(y)$$
is $C^{r,s}.$
\end{cor}

\begin{proof}
(a) \textit{Case $M_2$ a finite-dimensional manifold with corners.} Let $M_2$ be of dimension $n.$ Then each point of $M_2$ has an open neighbourhood homeomorphic to an open subset  $V$ of $[0, \infty[^n.$ Hence $V$ is locally compact, thus $M_2$ is locally compact. Thus Theorem \ref{k-mfd-iso} applies.\newline
(b) \textit{Case $E_1,\, E_2$ metrizable.} Then all points $x \in M_1, y \in M_2$ have open neighbourhoods $U_1 \subseteq M_1,\, U_2 \subseteq M_2$ homeomorphic to subsets $V_1 \subseteq E_1$ and $V_2 \subseteq E_2,$ respectively. Since $V_1 \times V_2$ is metrizable, it follows that $U_1 \times U_2 \times E_1 \times E_2$ is metrizable and hence a $k$-space. Hence by Lemma \ref{cap-ksp} $M_1\times M_2 \times E_1 \times E_2$ is a $k$-space and Theorem \ref{k-mfd-iso} applies.\newline
(c) \textit{Case $E_1$ and $E_2$ are $k_\omega$-spaces $M_1$ and $M_2$ are manifolds with corners.} For all $x\in M_1$ and $y\in M_2,$ there are open neighbourhoods $U_1 \subseteq M_1$,\, $U_2 \subseteq M_2$  homeomorphic to open subsets $V_1$ and $V_2$, respectively, of finite intersections of closed half-spaces in $E_1$ and $E_2$, respectively. Hence $V_1\times V_2 \times E_1 \times E_2$ is (relatively) open subset of a closed subset of $E_1\times E_2 \times E_1 \times E_2.$ The latter product is $k_\omega$ since $E_1$ and $E_2$ are $k_\omega$-spaces (see \cite[Proposition 4.2(i)]{ghh}), and hence a $k$-space.

Since open subsets (and also closed subsets) of $k$-spaces are $k$-spaces, it follows that  $V_1 \times V_2 \times E_1 \times E_2$ is a $k$-space. Now Lemma \ref{cap-ksp} shows that $M_1\times M_2 \times E_1 \times E_2$ is a $k$-space, and thus Theorem \ref{k-mfd-iso} applies.   

\end{proof}
\noindent
{\bf Proof for the comments after Theorem B.}
All assertions are covered by Corollary \ref{mm-iso}, except for the case when $M_1,\, M_2$ are manifolds with corners and $E_1\times E_2 \times E_1 \times E_2$ is a $k$-space. But this case can be proved like the result for $k_\omega$-spaces in Corollary \ref{mm-iso}.

\begin{rem} If $s=0,$ then $C^{r,s}$-maps $f\colon U \times V \to F$ can be defined just as well if $V$ is any Hausdorff topological space (and $U\subseteq E_1$ as before).\newline
If $r=0,$ then $C^{r,s}$-maps $f \colon U \times V \to F$ make sense if $U$ is a Hausdorff topological space. All results carry over to this situation (with obvious modifications).
\end{rem}

\begin{rem}
If $F$ is a complex locally convex space, we obtain analogous results if $E_1$ is a locally convex space over $\K_1 \in \{\R,\C \}$, $E_2$ is a locally convex space over $\K_2 \in \{\R,\C \},$ and all directional derivatives in the first and second variable are considered as derivatives over the ground field $\K_1$ and $\K_2,$ respectively, the corresponding maps could be called $C^{r,s}_{\K_1,\K_2}$-maps.
\end{rem}

\section{Parameter dependent differential equations} \label{sect: ODE}

To prove that diffeomorphism groups are regular Lie groups, one would like to solve certain differential equations depending on a parameter. Unfortunately, the parameter varies in a Fr\'{e}chet space, hence the usual theory of parameter dependent ODE's in Banach spaces does not suffice. We consider ODE's in a Banach space which depend on parameter sets in a locally convex space. In this setting, existence and uniqueness results are well known (see e.g., \cite[Section 10]{hg2006}). Our aim is to improve the differentiability properties of the flow if the right hand side of the ODE is a $C^{r,\infty}$-mapping. In the literature, the associated flows are mostly studied if the differential equation has a right hand side of class $C^r$.
For the reader's convenience we recall two facts, the first of which is a special case of {\cite[Theorem D]{hg2006}}:

\begin{lem}\label{lem: para:Cr}
 Let $(E,\norm{\cdot})$ be a Banach space, $U \subseteq E$ open and $Z$ be a locally convex space. Let $P \subseteq Z$ be open, $r\in \N_0 \cup \set{\infty}$ and $f \colon P \times U \rightarrow U$ be a $C^r$-map, such that the Lipschitz constants of $f$ in $U$ satisfy $\sup_{p \in P} \Lip (f(p,\bullet)) < 1$. Assume that for each $p \in P$, there is a fixed point $x_p$ of $f_p := f(p,\bullet) \colon U \rightarrow U$. Then $\varphi \colon P \rightarrow U, \varphi (p) := x_p$ is a $C^r$-map. 
\end{lem}

\begin{lem}\label{lem: pf:Ck}
 Let $K$ be a compact manifold (possibly with boundary), $E,F$ locally convex spaces, $r,s \in \N_0 \cup \set{\infty}$ and $U \subseteq E$ open. Consider a $C^{0,s}$-mapping $f \colon K \times U \rightarrow F$ such that $d^{(0,j)} f \colon K\times (U \times E^j) \rightarrow F$ is a $C^r$-map for all $j \in \N_0$ with $j \leq s$. Then $f_* \colon C^r(K,U) \rightarrow C^r(K,F), \gamma \mapsto f \circ (\id_K , \gamma)$ is of class $C^s$.
\end{lem}

\begin{proof}
 For manifolds without boundary a proof may be found in \cite[Proposition 3.10]{hg2002}. It is easy to see that the proof carries over without any changes to the general case.
\end{proof}

\begin{lem}\label{lem: parapf:Ck}
 Let $K$ be a compact manifold, $E,F,Z$ locally convex spaces, $r,s \in \N_0 \cup \set{\infty}$ and $U \subseteq E$, $P \subseteq Z$ open subsets. Consider a $C^{0,s}$-mapping 
    \begin{displaymath}
     f \colon K \times (U\times P) \rightarrow F
    \end{displaymath}
  such that $d^{(0,j)} f \colon K \times (U\times P) \times (E \times Z)^j \rightarrow F$ is a $C^r$-map for all $j \in \N_0$ such that $j \leq s$. Then 
  \begin{displaymath}
   \varphi \colon C^r(K,U) \times P \rightarrow C^r (K,F), (\gamma , p) \mapsto f(\bullet,p)_* (\gamma)
  \end{displaymath}
 is of class $C^s$.
\end{lem}

\begin{rem}
 The preceding lemma is closely connected to $C^{r,s}$ mappings. To emphasize this connection consider the following special cases: 
    \begin{compactenum}
     \item If $r=0$, then any $C^{0,s}$-map $f$ satisfies the requirements of \ref{lem: parapf:Ck}, hence \\ $\varphi \colon C(K,U)\times P \rightarrow C(K,F)$ is of class $C^s$. 
     \item If $k \geq r+s$ then any $C^{r,k}$-mapping $f$ satisfies the requirements, hence \\ $\varphi \colon C^r (K,U) \times P \rightarrow C^r(K,F)$ is a $C^s$-map.
    \end{compactenum}
\end{rem}

\begin{proof}[Proof of Lemma \ref{lem: parapf:Ck}]
 Denote by $\theta \colon Z \rightarrow C^r (K,Z)$ the continuous linear map which associates to each $p \in Z$ the mapping which takes $p$ as its only value. From Lemma \ref{lem: pf:Ck} we deduce that $f_* \colon C^r (K,U \times P) \rightarrow C^r (K,F)$ is a $C^s$-map. Taking the canonical isomorphism, we identify $C^r(K,U\times P)$ with $C^r (K,U) \times C^r (K,P)$ and obtain $\varphi (\gamma, p) = f_* (\gamma ,  \theta (p))$. Hence $\varphi$ is a $C^s$-map, as desired.
\end{proof}

The general setting we shall be working in is as follows:

\begin{setup}\label{setup1} Let $(E,\norm{\cdot})$ be a Banach space, $r\in \N_0 \cup \set{\infty}$, $s,k \in \N \cup \set{\infty}$ with $k \geq r+s$. Furthermore let $J \subseteq \R$ be a non degenerate interval, $F$ be a locally convex space and $P \subseteq F, U \subseteq E$ be open subsets. Consider a map $f \colon J \times (U \times P) \rightarrow E$ which is $C^{r,k}$ with respect to $J\times (U\times P)$.
\end{setup} 

We prove a parameter-dependent version of the Picard Lindelöf theorem (cf.\ also \cite[Theorem 10.3]{hg2006}):

\begin{thm}\label{thm: A:ODES} 
 In the setting of \ref{setup1}, choose $t_0 \in J$, $x_0 \in U$ and $p_0 \in P$. Then there exist a convex neighbourhood $J_0 \subseteq J$ of $t_0$ and open neighbourhoods $U_0 \subseteq U$ of $x_0$ and $P_0 \subseteq P$ of $p_0$ such that for all $(\tau_0,y_0,q_0) \in J_0 \times U_0 \times P_0$, the initial value problem 
  \begin{equation}\label{eq: ODE}
   \begin{cases}
    x' (t)     &= f(t,x(t),q_0) \\
    x (\tau_0) &= y_0
   \end{cases}
  \end{equation}
 has a unique solution $\varphi = \varphi_{\tau_0,y_0,q_0} \colon J_0 \rightarrow U$. Furthermore, the map
  \begin{displaymath}
   \Phi \colon J_0 \times J_0 \times (U_0 \times P_0) \rightarrow E, (\tau_0,t,(y_0,q_0)) \mapsto \varphi_{\tau_0,y_0,q_0} (t) 
  \end{displaymath}
 yields $C^{r+1,s}$-maps $\Phi_{\tau_0} \coloneq \Phi (\tau_0 , \bullet) \colon J_0 \times (U_0 \times P_0) \rightarrow E$ for each $\tau_0 \in J_0$.
 \end{thm}

\begin{proof}
 Consider the continuous mapping 
  \begin{displaymath}
   \omega \colon J \times U \times P \times E \rightarrow E , (t,x,p,h) \mapsto d^{(0,1)} f (t,x,p; (h,0))
  \end{displaymath}
 which satisfies $\omega (t_0,x_0,p_0,0)=0$. We deduce that there is a connected neighbourhood $I$ of $t_0$ in $J$, an open neighbourhood $Q \subseteq P$ of $p_0$ and $R >0$ such that the following holds: The open ball $B_{2R}^E (x_0)$ is contained in $U$ and $\omega (I \times B_{2R}^E (x_0) \times Q \times B_R^E (0)) \subseteq B_1^E (0)$. This yields the estimate \begin{equation}\label{eq: Lip}
          \norm{\omega (t,y,p, \bullet)}_{\text{op}} \leq \frac{1}{R} \equalscolon L        \quad \forall (t,y,p) \in I \times B_{2R}^E (x_0) \times Q   .                                                                                                                                                                                                                                                                                                         
 \end{equation}
 As $B_{2R}^E (x_0)$ is convex, equation \eqref{eq: Lip} proves $\Lip f(t,\bullet , p)|_{B_{2R}^E (x_0)} \leq L , \forall (t,p) \in I \times Q$, i.e. a Lipschitz condition in the Banach space component. By shrinking $R$ and $I$ and choosing an open neighbourhood $P_0 \subseteq Q$ of $p_0$ we may assume that $I$ is compact and $f(I \times B_{2R}^E (x_0) \times P_0) \subseteq B_M^E (0)$ holds for some $M \in ]0,\infty[$. Choose $\varepsilon > 0$ with $\varepsilon < \min \set{\tfrac{R}{M} , \tfrac{1}{L}}$ and define $J_0 \coloneq \set{t \in I \colon \lvert t-t_0 \rvert \leq \ve}$, $U_0 \coloneq B_{R}^E (x_0)$. Observe that $J_0$ is a compact set. Our arguments will yield solutions for each $\tau_0 \in J_0$. To ease notation choose and fix $\tau_0 \in J_0$. Now 
  \begin{displaymath}
   g \colon J_0 \times (B_R^E (0) \times U_0 \times P_0) \rightarrow E , g(t,x,y,p) \coloneq f(t,x+y,p)
  \end{displaymath}
 is a $C^{r,k}$-mapping by the Chain Rule \ref{chain-1}. Clearly $\varphi$ will solve the initial value problem \eqref{eq: ODE} if and only if $\Psi \coloneq \varphi - y_0$ solves \begin{equation}\label{eq: ODE2}
          \begin{cases}
           \Psi' (t) = g(t,\Psi (t),y_0,q_0) \\
	   \Psi (\tau_0) =0
          \end{cases}.
 \end{equation}
 By the Fundamental Theorem of Calculus for continuous $\Psi \colon J_0 \rightarrow  B_R^E (0)$, the initial value problem \eqref{eq: ODE2} is equivalent to the integral equation 
    \begin{equation}\label{eq: int}
     \forall t \in J_0 \quad \Psi (t) = \int_{\tau_0}^t g(s,\Psi (s),y_0,q_0)\, ds.
    \end{equation}
  Now $X \coloneq \setm{\gamma \in C (J_0,E)}{\gamma (\tau_0) =0}$ is a closed vector subspace of the Banach space $C(J_0,E)$, hence a Banach space. The map  
  \begin{displaymath}
		h \colon U_0 \times P_0 \times B_R^X (0) \rightarrow X , h(y,p,\gamma) (t) \coloneq \int_{\tau_0}^t g (s,\gamma (s),y,p)\, ds
  \end{displaymath}
  is well defined by the Fundamental Theorem. Indeed $\im (h) \subseteq B_R^X (0)$ holds, as 
    \begin{align}\label{eq: est}
      \sup_{t \in J_0} \norm{h(y,p,\gamma) (t)} &\leq \lvert t -\tau_0\rvert \sup_{s \in J_0} \norm{g(s,\gamma (s), y,p)}  \notag\\
						&\leq \ve \sup_{s \in J_0} \norm{f(s,y + \gamma (s),p)} \leq \ve M < R  .
    \end{align}
  We want to apply Banach's Contraction Theorem, hence $h$ has to define a uniform family of contractions, i.e.\ $\sup_{(y,p) \in B_R^E(x_0) \times P_0} \Lip (h(y,p,\bullet)) < 1$. To prove this we compute
    \begin{align*}
     \norm{h (y,p,\gamma_1)(t) - h(y,p,\gamma_0) (t)} &= \norm{\int_{\tau_0}^t g(s,\gamma_1 (s),y,p) - g(s,\gamma_0 (s), y,p)\,ds} \\
						      &= \norm{\int_{\tau_0}^t f(s,y+\gamma_1 (s),p)-f(s,y+\gamma_0 (s),p)\,ds}\\
						      &\leq \lvert t - \tau_0 \rvert \cdot L \cdot \sup_{s \in J_0} \norm{\gamma_1 (s) - \gamma_0 (s)} \leq \ve L \norm{\gamma_1 - \gamma_0}_{\infty}
    \end{align*}
  and observe that by choice $\ve L < 1$ holds. From equation \eqref{eq: est} we derive the estimate $h(U_0 \times P_0 \times \overline{B_{\ve M}^X (0)}) \subseteq \overline{B_{\ve M}^X (0)}$. As $\overline{B_{\ve M}^X (0)}$ is a complete metric space, Banach`s Contraction Theorem shows that there is a unique fixed point $\Psi_{\tau_0,x_0,p_0} \in \overline{B_{\ve M}^X (0)}$ of $h(x_0,p_0, \bullet)$. We may view $\Psi_{\tau_0,x_0,p_0}$ as an element of $B_R^X (0)$. Retracing our steps, equation \eqref{eq: int} implies that $\Psi_{\tau_0,x_0,p_0}$ is the unique solution to \eqref{eq: ODE2} and $\varphi_{\tau_0,x_0,p_0} (\bullet) \coloneq \Psi_{\tau_0,x_0,p_0} (\bullet) + x_0$ is a solution to \eqref{eq: ODE}. 
\paragraph{} By the above, the existence of $\Phi$ on $J_0 \times J_0 \times U_0 \times P_0$ is clear. We have to prove the differentiable dependence of the solution on time, initial value and its parameter for fixed $\tau_0$. To achieve this, we check that $h$ is a $C^s$-map. The map $g$ is of class $C^{r, k}$ on $J_0 \times (B_R (0) \times B_R (x_0) \times P_0)$ with $J_0$ compact. Interpreting $B_R^E(x_0) \times P$ as a set of parameters, Lemma \ref{lem: parapf:Ck} shows that  
    \begin{displaymath}
     \Gamma \colon U_0 \times P_0 \times C(J_0,B_R^E (0)) \rightarrow C(J_0,E) , (y,p,\gamma) \mapsto g(\bullet , y,p)_* (\gamma )
    \end{displaymath}
 is a $C^k$-map. Furthermore, $S \colon C(J_0,E) \rightarrow X, S(\gamma) (t) \coloneq \int_{\tau_0}^t \gamma (s) ds$ is continuous linear, thus smooth. Since $h = S \circ \Gamma$ we deduce that $h$ is $C^s$.\\
 The map $\Psi_{\tau_0,y,p}$ is the fixed point of $h(y,p,\bullet)$. But $h \colon (U_0 \times P_0) \times B_R^X (0) \rightarrow X$ is $C^s$ and $B_R^X (0) \subseteq X$ is an open subset of a Banach space $X$. By Lemma \ref{lem: para:Cr}, the map $\Psi_{\tau_0} \colon B_R^E (x_0) \times P_0 \rightarrow C(J_0,E), (y,p) \mapsto \Psi_{\tau_0,y,p}$ is $C^s$. Consider the continuous linear map $\theta \colon E \rightarrow C(J_0,E)$, which assigns to each $x \in E$ the constant map with image $x$. Recall that $\varphi_{\tau_0,y,p} = \Psi_{\tau_0,y,p} + \theta (y)$ holds. By construction, $\varphi_{\tau_0,y,p}$ solves the initial value problem \eqref{eq: ODE}, i.e. $\varphi'_{\tau_0,y,p} (t) = f(t,\varphi_{\tau_0,y,p} (t),p)$, $\forall (t,y,p) \in J_0 \times U_0 \times P_0$. Thus a $C^s$-map with values in $C^1(J_0,E)$ is given by 
  \begin{displaymath} 
   \Theta_{\tau_0} \colon U_0 \times P_0 \rightarrow C(J_0,E), (y,p) \mapsto \varphi_{\tau_0,y,p} = \Psi_{\tau_0} (y,p) + \theta (y).
  \end{displaymath} 
Define $D \colon C^1 (J_0,E) \rightarrow C(J_0,E), \gamma \mapsto \gamma'$. We compute:\\ $(D \circ \Theta_{\tau_0}) (y,p) (t) = \varphi_{\tau_0,y,p}' (t) = f(t, \varphi_{\tau_0,y,p} (t),p) = g(t, \Psi_{\tau_0,y,p} (t), y,p)$. Hence  
  \begin{displaymath}
   (D \circ \Theta_{\tau_0}) (y,p) = \Gamma (y,p,\Psi_{\tau_0} (y,p)),
  \end{displaymath}
 whence $D \circ \Theta_{\tau_0}$ is a $C^s$-map. The mapping $\Lambda \colon C^1(J_0,E) \rightarrow C(J_0,E)\times C(J_0,E) , \gamma \mapsto (\gamma, D(\gamma))$ is a linear topological embedding with closed image by Lemma \ref{lem: cr:emb}. Combining the results from above, $\Theta_{\tau_0} \colon U_0 \times P_0 \rightarrow C^1 (J_0,E)$ is a $C^s$-map as $\Lambda \circ \Theta_{\tau_0}$ is $C^s$. We now prove by induction on $j \in \set{0, \ldots , r}$ that $\Theta_{\tau_0} \colon U_0 \times P_0 \rightarrow C^{j+1} (J_0,E)$ is a $C^s$-map. Having already dealt with the case $j = 0$, assume that $j > 0$ and $\Theta_{\tau_0} \colon U_0 \times P_0 \rightarrow C^j (J_0,E)$ is $C^s$. As $f$ is a mapping of class $C^{r,k}$ with $k \geq r \geq j$, it is a $C^j$-map by Lemma \ref{lem: C11:C1}. We deduce from \eqref{eq: ODE} that $\varphi_{\tau_0,y,p}'$ is a map of class $C^j$, whence $\varphi_{\tau_0,y,p}$ is of class $C^{j+1}$. Therefore $\im (\Theta_{\tau_0}) \subseteq C^{j+1} (J_0,E)$ and we are left to prove 
the $C^s$-property of $\Theta_{\tau_0}$ as a map to $C^{j+1} (J_0,E)$. To this end, define $D \colon C^{j+1} (J_0,E) \rightarrow C^j (J_0,E), \gamma \mapsto \gamma'$ and $\Gamma^j \colon C^j (J_0, B_R^E(0)) \times (U_0 \times P_0)  \rightarrow C^j (J_0,E), (\gamma , (y,p)) \mapsto (g(\bullet, y,p))_* (\gamma)$. Again as $g$ is a $C^{r,k}$-map, $\Gamma^j$ is $C^s$ by Lemma \ref{lem: parapf:Ck}. One easily checks that $D \circ \Theta_{\tau_0} = \Gamma^j \circ (\Psi_{\tau_0} , \id)$ holds and thus $D \circ \Theta_{\tau_0}$ is $C^s$. By Lemma \ref{lem: cr:emb}, $\Lambda^{j+1} \colon C^{j+1} (J_0,E) \rightarrow C(J_0,E) \times C^j (J_0,E), \gamma \mapsto (\gamma , D(\gamma))$ is a linear topological embedding with closed image. As $\Lambda^{j+1} \circ \Theta_{\tau_0}$ is $C^s$ by the above, $\Theta_{\tau_0}$ is $C^s$ as a map to $C^{j+1} (J_0,E)$. This concludes the induction step, proving $\Theta_{\tau_0} \colon U_0 \times P_0 \rightarrow C^{r+1} (J_0,E)$ to be a $C^s$-map.  
 Since $J_0$ is a compact, hence a locally compact space, Theorem \ref{kspace-iso} (a) proves $(\Theta_{\tau_0})^\wedge \colon (U_0 \times P_0) \times J_0 \rightarrow E$ to be a mapping of class $C^{s , r+1}$. Corollary \ref{rs-sr} implies that $\Phi_{\tau_0} \colon J_0 \times (B_R^E(0) \times P_0) \rightarrow E , \Phi_{\tau_0} (t,(y,p)) = \varphi_{\tau_0,y,p} (t) = (\Theta_{\tau_0})^\wedge ((y,p),t)$ is a $C^{r+1,s}$-map.
\end{proof}

\begin{rem}
 \begin{compactenum}
  \item Notice that if $k < \infty$ in the last theorem, we lose differentiability orders of the solution with respect to initial value and parameter. 
  \item There is an alternative proof for the differentiable dependence of the solution: Prove the differentiable dependence by induction. To do so exploit the ``mixed partial derivatives`` outlined in \cite[II. Theorem 9.2]{amann1990} combined with the Chain Rule \ref{lem: chainrule}. Notice however that one still needs $f$ to be a $C^{r,k}$-mapping with $k \geq r+s$, as also the Chain Rule  \ref{lem: chainrule} decreases the order of differentiability.
 \end{compactenum}
\end{rem}

The last theorem provided a local uniqueness and existence result. Global results for the flow of a differential equation follow in the wash. 

\begin{rem}\label{rem: Flow}
  Theorem \ref{thm: A:ODES} proves that in the situation of \ref{setup1} with $(t_0,x_0,p_0) \in J \times U \times P$ there is a unique solution in a neighbourhood of $t_0$. Fix $t_0 \in J$. Applying the usual arguments (cf. \cite[II. 7.6]{amann1990}), the uniqueness assertion in \ref{thm: A:ODES} allows one to construct a unique maximal solution to \eqref{eq: ODE} with $x(t_0) = x_0$ and parameter $p_0 \in P$. By abuse of notation, we denote the maximal solutions by $\varphi_{t_0,x_0,p_0}$. It coincides with the maps constructed in Theorem \ref{thm: A:ODES}. The solution $\varphi_{t_0,x_0,p_0}$ is defined on an interval $J^{t_0}_{x_0,p_0}$, which is a neighbourhood of $t_0$ in $J$. Again as in \cite[II. 7.6]{amann1990}, $J^{t_0}_{x_0,p_0}$ is an open subset of $J$.\\
  We introduce the set $\mathfrak{D}(f) \coloneq \bigcup_{t_0 \in J, x_0 \in U , p_0 \in P} \set{t_0} \times J^{t_0}_{x_0,p_0} \times \set{(x_0,p_0)}$ and define the \emph{flow} of the differential equation:  $\Fl^f \colon \mathfrak{D} (f) \rightarrow U, (t_0,t,(x_0,p_0)) \mapsto \varphi_{t_0,x_0,p_0} (t)$. it is known that $\mathfrak{D}(f)$ is open in $J \times J \times U \times P$ and $\Fl^f$ is continuous (cf. \cite{hg2006} for the local argument), but we shall not presume this. In the remainder of this section we shall study differentiability properties of $\Fl^f_{t_0} \coloneq \Fl^f (t_0,\bullet)$.  
\end{rem}

\begin{prop}\label{prop: flow:diff} 
 In the setting of Theorem \ref{thm: A:ODES} fix $t_0 \in J$. If $f$ is of class $C^{r,s}$, then $\Fl^f_{t_0}$ is a $C^{r+1,s}$-map on the open subset $\Omega_{t_0} \coloneq \bigcup_{x_0 \in U , p_0 \in P} J^{t_0}_{x_0,p_0} \times \set{(x_0,p_0)} \subseteq J \times (U \times P)$.
\end{prop}

\begin{proof}
 Consider arbitrary $(t_0,x_0,p_0) \in J \times U \times P$. Let $J^\star_{x_o,p_0}$ be the set of points $a \in J^{t_0}_{x_0,p_0}$ for which the following conditions hold
    \begin{compactenum}
     \item There is a relatively open convex neighbourhood $I_a \subseteq J$ of $a$ together with an open neighbourhood $V_a \subseteq U \times P$ of $(x_0,p_0)$, such that $I_a \times V_a \subseteq \Omega_{t_0}$;
     \item The restriction $(\Fl^f_{t_0})|_{I_a \times V_a}$ is a $C^{r+1,s}$-mapping.
    \end{compactenum}
  Then $J^\star_{x_0,p_0}$ is an open subset of $J^{t_0}_{x_0,p_0}$ whose interior contains $t_0$, by Theorem \ref{thm: A:ODES}. 
  \paragraph{Claim:} $J^\star_{x_0,p_0} = J^{t_0}_{x_0,p_0}$.  If this is correct, then for each $(t',x,p) \in \Omega_{t_0}$ we derive $t' \in J^{t_0}_{x,p} = J^\star_{x,p}$. The definition of $J^\star_{x,p}$ implies that $\Omega_{t_0}$ is then a neighbourhood of $(t',x,p)$. Thus $\Omega_{t_0}$ is open (as a subset of $J \times U \times P$). Furthermore $\Fl^f_{t_0}$ is of class $C^{r+1,s}$, since this condition is satisfied locally by property (b) of $J^\star_{x_0,p_0}$.\\
  The set $J^{t_0}_{x_0,p_0}$ is convex, hence connected. To obtain $J^\star_{x_0,p_0} = J^{t_0}_{x_0,p_0}$ it suffices to show that $J^\star_{x_0,p_0}$ is also a closed subset of $J^{t_0}_{x_0,p_0}$.\\
  Denote by $t^* \in \partial J^\star_{x_0,p_0}$ a point in the boundary with respect to $J^{t_0}_{x_0,p_0}$. Without loss of generality we may assume $t^* > t_0$ (the argument for $t^* < t_0$ is analogous). By Theorem \ref{thm: A:ODES} there are a convex open subset $t^* \in J_{t^*} \subseteq J^{t_0}_{x_0,p_0}$ and open neighbourhoods $W_1 \subseteq U$ of $y_0 \coloneq \Fl_{t_0}^f (t^*,x_0,p_0)$, resp.\ $W_2 \subseteq P$ of $p_0$, such that:\\
  There is a map $\Phi^* \colon J_{t^*} \times J_{t^*} \times (W_1 \times W_2) \rightarrow U$ with $\Phi_a^* \coloneq \Phi^* (a,\bullet)$ being of class $C^{r+1,s}$ for each $a \in J_{t^*}$. The map $\Phi^*_a (\bullet,y,p) $ solves \eqref{eq: ODE} and is thus continuous with $\Phi^*_{a} (a,x,p) = x$, for all $(a,x,p) \in J_{t^*} \times W_1 \times W_2$.\\
  We already know $\Fl_{t_0}^f (\bullet, x_0,p_0) = \varphi_{t_0,x_0,p_0} (\bullet)$ to be continuous by Remark \ref{rem: Flow}. As its domain contains $t^*$, there is $t_0 < a < t^*$ with $a \in J_{t^*}$ and $\Fl^f_{t_0} (a,x_0,p_0) \in W_1$. Since $t_0 < a < t^*$ holds, $a$ is contained in $J^\star_{x_0,p_0}$. Thus there exist convex open subsets $a \in I_a \subseteq J^{t_0}_{x_0,p_0}$, an open neighbourhood $V_1 \subseteq U$ of $x_0$ and an open neighbourhood $V_2 \subseteq W_2$ of $p_0$ such that $\Fl^f_{t_0}$ is of class $C^{r+1,s}$ on $I_a \times V_1 \times V_2$. By continuity we may shrink the sets, such that $\Fl_{t_0}^f (I_a \times V_1 \times V_2) \subseteq W_1$ and $I_a \subseteq J_{t^*}$ holds. Now for all $(t,x,p) \in I_{a} \times V_1\times V_2$, by construction $(\Fl_{t_0}^f (t,x,p),p) \in W_1 \times W_2$.\\
  Let $(x,p) \in V_1 \times V_2$. The maps $\Fl_{t_0}^f (\bullet , x , p)$ and $\Phi^* (a, \bullet , \Fl_{t_0}^f (x,p),p)$ are both solutions to the initial value problem
    \begin{displaymath}
     \begin{cases}
      y'(t) &= f(t,y(t),p), \\
      y(a)  &= \Fl_{t_0}^f (a,x,p) .
     \end{cases}
    \end{displaymath}
 The uniqueness assertion in Theorem \ref{thm: A:ODES} then implies 
  \begin{equation}\label{eq: new}
   \Fl_{t_0}^f (t,x,p) = \Phi^* (a, t,\Fl_{t_0}^f (a,x,p),p) \text{ for } (t,x,p) \in I_a \times V_1 \times V_2. 
  \end{equation}
 By construction, the right-hand side of \eqref{eq: new} is defined on $J_{t^*}$. Furthermore by the Chain Rule \ref{chain-1}, it is a mapping of class $C^{r+1,s}$ on $J_{t^*} \times (V_1 \times V_2)$. Thus $\Fl^f_{t_0}$ is of class $C^{r+1,s}$ on this product. But the product $J_{t^*} \times V_1 \times V_2$ contains $(t^*, x_0, p_0)$ in its interior and $\Fl^f_{t_0}$ is $C^{r+1,s}$ on this set. Thus $t^* \in J^\star_{x_0,p_0}$ holds by definition of $J^\star_{x_0,p_0}$ and we derive that $J^\star_{x_0,p_0} = \overline{J^\star_{x_0,p_0}}$. 
 \end{proof}

\begin{rem}
 The statement of Proposition \ref{prop: flow:diff} yields the following additional information: $\Omega_{t_0}$ is a locally convex subset with dense interior of $\R \times U \times P$ if $J$ is a non-open interval. If $J$ is open, then $\Omega_{t_0}$ is an open subset of $\R \times U \times P$. 
\end{rem}

To state our last result, we need $C^{r,s}$-mappings between manifolds.

\begin{defn}\label{defn: crs:mfd}
 Let $M_i, i \in \set{1,2,3}$, be smooth manifolds (possibly with rough boundary) modelled on locally convex spaces $E_i$. A map $f \colon M_1 \times M_2 \rightarrow M_3$ is called a $C^{r,s}$-mapping, if it is continuous and for all charts $\varphi$ of $M_1$, $\psi$ of $M_2$ and $\kappa \colon V_\kappa \rightarrow U_\kappa$ of $M_3$ the mapping 
      \begin{displaymath}
       \kappa \circ f \circ (\varphi^{-1} \times \psi^{-1})|_{\varphi \times \psi (f^{-1} (V_\kappa))} \colon \varphi \times \psi (f^{-1} (V_\kappa) \cap U_\phi \cap U_\psi) \rightarrow E_3
      \end{displaymath}
 is a $C^{r,s}$-mapping.\\
 As all manifolds involved are smooth, the change of chart maps are smooth. It readily follows from the chain rules (\ref{chain-1} and \ref{chain-2}) that it suffices to check the $C^{r,s}$-property for arbitrary atlases of $M_1$, $M_2$ and $M_3$.
\end{defn}

\begin{setup}\label{setup: mfd}
 Let $J$ be a non-degenerate interval, $M$ a smooth manifold (without boundary) modelled on a Banach space $E$, $r \in \N_0 \cup \set{\infty}$, $k,s \in \N \cup \set{\infty}$ with $k \geq r+s$ and $P \subseteq Z$ an open subset of the locally convex space $Z$.\\ Assume that $X \colon J \times (M\times P) \rightarrow TM, (t,(x,p)) \mapsto X_{t,p} (x)$ is a map of class $C^{r,k}$ such that $X_{t,p}(x) \in T_xM$ holds for each $(t,x,p) \in J \times M \times P$. Let $t_0 \in J$, $(y_0 , p_0) \in M \times P$. Using local representatives of $X$ together with the local existence and uniqueness result (Theorem \ref{thm: A:ODES}), one may apply standard arguments (cf. \cite[IV. \S 2]{langdgeo2001}) to obtain a maximal solution $\varphi_{t_0,y_0,p_0} \colon I_{t_0,y_0,p_0} \rightarrow M$ to the differential equation 
      \begin{equation}\label{eq: mfd:ODE}
       \begin{cases}
        \varphi' (t) &= X (t,\varphi (t),p_0) = X_{t,p_0} (\varphi (t))\\ 
	\varphi (t_0)&= y_0
       \end{cases}    
      \end{equation}
 defined on an open subset $I_{t_0,y_0,p_0} \subseteq J$.
\end{setup}

 \begin{prop}\label{prop: flow:mfd}
  In the setting of \ref{setup: mfd}, define the subset 
  \begin{displaymath}
   \Omega \coloneq \displaystyle \bigcup_{(t_0,y_0,p_0) \in J\times M\times P} \set{t_0} \times I_{t_0,y_0,p_0} \times \set{y_0,p_0}
  \end{displaymath}
  of $J \times J \times M \times P$ and the map 
  \begin{displaymath}
   \Phi \colon \Omega \rightarrow M, \Phi (t_0,t,y_0,p_0) \coloneq \varphi_{t_0,y_0,p_0} (t).
  \end{displaymath}
 Then for $t_0 \in J$, the set $\Omega_{t_0} \coloneq \bigcup_{(y_0,p_0) \in M\times P} I_{t_0,y_0,p_0} \times \set{y_0,p_0}$ of $J \times M \times P$ is an open subset of $J \times M \times P$ and $\Phi_{t_0} \coloneq \Phi (t_0,\bullet)$ is a mapping of class $C^{r+1,s}$ with respect to $\Omega_{t_0} \subseteq J \times (M\times P)$.
 \end{prop}

\begin{proof}
 We begin with a local computation in charts: Let $\kappa \colon M \supseteq V_\kappa \rightarrow U_\kappa \subseteq E$ be some chart for $M$. Then the local representative 
  \begin{displaymath}
   X_\kappa \colon J \times (U_\kappa \times P) \rightarrow E, X_\kappa \coloneq \text{pr}_E \circ T \kappa \circ X \circ (\id_J \times \kappa^{-1} \times \id_P)
  \end{displaymath}
 (where $\text{pr}_E$ is the canonical projection onto $E$) is a $C^{r,k}$-map by construction. Using the local representative and Theorem \ref{thm: A:ODES}, one easily obtains the following:\\ 
 For each $(t,x,p) \in J \times V_\kappa \times P$ there is a relatively open neighbourhood $t \in I^\kappa_{t,x,p} \subseteq J$ and open subsets $x \in V_1 \subseteq V_\kappa$, $p \in V_2 \subseteq P$ such that  $I^\kappa_{t,x,p} \times I^\kappa_{t,x,p} \times (V_1 \times V_2) \subseteq \Omega$. For each $s \in I^\kappa_{t,x,p}$, the mapping $\Phi_s \rvert_{I^\kappa_{t,x,p} \times (V_1 \times V_2)}$ is a mapping of class $C^{r+1,s}$.\\
 Observe that this local flow is uniquely determined as Theorem \ref{thm: A:ODES} ensures uniqueness of the solution in a small neighbourhood. Reviewing the proof of Proposition \ref{prop: flow:diff}, only local existence of the solution and local uniqueness were used. Thus we may repeat the argument given in the proof of Proposition \ref{prop: flow:diff} to obtain the desired result.  
\end{proof}
\noindent
{\bf Acknowledgement and remark:} The authors would like to express their deep gratitude to Helge Gl\"ockner for his invaluable assistance and the helpful discussions concerning this work. Moreover, we thank the referee for many helpful comments which helped improve the manuscript. Sections 3 and 4 were written by the first author, and will be part of his Ph.D.-thesis~\cite{al}. Section 5 (and Lemma \ref{lem: chainrule}) were written by the second author, and will be used in his Ph.D-thesis~\cite{as2013}.

\bibliographystyle{babplain-lf}

\phantomsection
\addcontentsline{toc}{section}{References}

\bibliography{Alzaareer_Schmeding_CRS}
\end{document}